\documentclass[12pt]{article}

\usepackage[letterpaper,margin=1in]{geometry}
\usepackage{amsmath,amssymb,mathtools}

\usepackage[amsmath,amsthm,thmmarks,hyperref]{ntheorem}

\usepackage{etoolbox}
\AtBeginEnvironment{proof}{%
  \theorempreskipamount=6pt plus6pt\relax
  \theorempostskipamount=6pt plus6pt\relax}
\makeatletter
\apptocmd{\@endtheorem}{\@endpefalse}{}{}

\patchcmd{\SetTagPlusEndMark}
  {\llap{\hss\PotEndMark{\raisebox{-1.3em}}}}
  {\PotEndMark{\quad}}
  {}{\PackageError{proofmarks}{Cannot configure automatic proof marks}{}}
\makeatother

\usepackage{booktabs,tabularx,array,dcolumn}
\usepackage{graphicx}
\usepackage{float}
\usepackage{flafter}

\makeatletter
\let\ftype@table\ftype@figure
\makeatother
\usepackage[labelsep=period]{caption}
\newlength{\reconstructioncaptionwidth}
\usepackage{placeins}
\usepackage{enumitem}
\usepackage{microtype}
\usepackage{xcolor}
\usepackage[hidelinks]{hyperref}

\numberwithin{equation}{section}
\allowdisplaybreaks
\newtheorem{theorem}{Theorem}[section]
\newtheorem{proposition}[theorem]{Proposition}
\newtheorem{lemma}[theorem]{Lemma}
\newtheorem{corollary}[theorem]{Corollary}

\theoremstyle{definition}

\newtheorem{assumption}[theorem]{Assumption}

\newtheorem{remark}[theorem]{Remark}

\setlist[enumerate]{leftmargin=2.2em,itemsep=0.45em,topsep=0.35em}
\setlist[itemize]{leftmargin=2em,itemsep=0.25em,topsep=0.3em}
\newcolumntype{Y}{>{\raggedright\arraybackslash}X}

\newcommand{\R}{\mathbb{R}}
\newcommand{\A}{\mathcal{A}}
\newcommand{\J}{\mathcal{J}}
\newcommand{\Tr}{\operatorname{Tr}}
\newcommand{\diver}{\operatorname{div}}
\newcommand{\clip}{\operatorname{clip}}

\title{Neural Stream Function Method for Electrical Impedance Tomography
with Lipschitz Conductivities}

\author{
Xiaosheng Li \\[0.5em]
{\normalsize Department of Mathematics and Statistics}\\
{\normalsize Florida International University}\\
{\normalsize Miami, FL 33199, U.S.A.}\\
{\normalsize Email: \texttt{xli@fiu.edu}}
}

\date{}

\begin{document}

\maketitle

\begin{abstract}
Electrical impedance tomography (EIT) reconstructs the conductivity inside
an object from voltage and current measurements on its boundary. The
resulting inverse problem is nonlinear and severely ill-posed. Neural
networks offer a flexible computational approach, but training physics
informed neural networks for EIT and relating their training loss to
reconstruction accuracy remain challenging. We develop a novel
constrained neural stream function method for planar EIT with Lipschitz conductivities.
The method builds the conductivity constraints and prescribed boundary
voltages into the neural trial functions and replaces the second order
conductivity equation with a first order system in strong form. The
measured current is also used to construct a fixed baseline for the stream
function. For this formulation, we derive error estimates that relate the sampled neural 
objective to errors in the physical boundary currents and the reconstructed conductivity.
The error estimates separate the contributions from neural
approximation, measurement noise, and sampling. 
Numerical experiments further demonstrate that the proposed
method achieves high reconstruction accuracy with both exact and noisy
boundary data.
\end{abstract}


\section{Introduction}
\label{sec:introduction}

Electrical impedance tomography (EIT) recovers the electrical conductivity
inside an object from voltage and current measurements on its boundary.
This noninvasive technique has a wide range of applications, such as
breast cancer and pneumothorax detection, lung function monitoring,
industrial process monitoring, and leak detection in buried pipes and
underground storage tanks
\cite{CostaEtAl2008,IsaacsonEtAl1990,JordanaEtAl2001,RamirezEtAl1996,Wang2015,ZouGuo2003}. The conductivity reconstruction is a nonlinear and
severely ill-posed inverse boundary value problem. Since Calder\'on's
original work \cite{Calderon1980}, its uniqueness and stability have
been studied extensively
\cite{Alessandrini1988,AlessandriniVessella2005,AstalaPaivarinta2006,BarceloFaracoRuiz2007,Mandache2001,Nachman1996,SylvesterUhlmann1987,Uhlmann2009}. In computations, only finitely
many boundary measurements are available, and the data may contain noise.
A variety of reconstruction methods have been developed to address these
difficulties, including direct reconstruction
\cite{ChowItoZou2014,KirschGrinberg2008,KnudsenEtAl2009,MuellerSiltanen2020},
regularized iteration methods
\cite{Borcea2002,CheneyEtAl1990,Dobson1992,Hanke1997,JinMaassSparsity2012,LechleiterRieder2006,LiWang2025},
variational methods
\cite{BorceaGrayZhang2003,KohnVogelius1987}, and statistical methods
\cite{KaipioSomersalo2005}.

In recent years, deep learning methods for EIT reconstruction have
received considerable attention. Many of these methods rely on precomputed
training data. Supervised methods learn a map from simulated boundary
measurements to conductivity images
\cite{AdlerGuardo1994,ChenEtAl2020,FanYing2020,LiEtAl2017,SeoEtAl2019}.
Other approaches combine learned components with D-bar reconstruction,
Calder\'on's method, direct sampling, or iterative reconstruction methods
\cite{BhatChenWang2026,CenJinShinZhou2023,ChenEtAl2023,ColibazziEtAl2023,GuoJiang2021,HamiltonEtAl2019,HerzbergEtAl2021}. Neural operators have also been used
to approximate the continuum EIT inverse map
\cite{deHoopEtAl2025,SchotlandTitiWang2026}. We refer to
\cite{CebrianCarcavillaMeribout2025,DenkerEtAl2025} and the references therein 
for some reviews of neural EIT methods. 
Once trained, these models can provide fast reconstructions. Their accuracy, however, may decrease when the conductivity being reconstructed is not well represented in the training data. Generating a representative training dataset and training the models also require upfront computational effort.

Some other approaches use iterative reconstruction methods that represent
only the unknown conductivity by a neural network and still rely on the
finite element method (FEM) to solve the forward problem
\cite{LiuEtAl2023,ParkEtAl2026}. At each iteration, the network parameters
are updated using the discrepancy between the measured boundary data and
those computed from the conductivity predicted by the network. These
approaches therefore still require repeated FEM forward solves during
training.

The physics informed neural network (PINN) approach has emerged as a
powerful and flexible framework for solving partial differential equations
\cite{RaissiEtAl2019}. It approximates the unknown fields with neural
networks and trains these approximations
using the governing equations, boundary conditions, and available
observations. Several EIT formulations of this type use interior
information or a separately trained model. The EIT computations in
\cite{BarSochen2021} mainly address the setting in which the interior
potentials are observed. The approach in \cite{PokkunuruEtAl2023}
considers the same setting and uses a prior based on an energy model
trained on an additional data set. For the full inverse EIT problem,
the method in \cite{YangEtAl2026} uses a pretrained convolutional neural
network to estimate the interior potentials from boundary measurements
and then uses a PINN to reconstruct the conductivity.

These considerations motivate fully neural, physics informed approaches 
to inverse EIT, in which both the conductivity and the interior potentials are 
represented by neural networks. The networks are trained using only
the boundary measurements and the governing equation, without additional
training data or reliance on conventional forward solvers.
The method in \cite{KaloutEtAl2026} uses the strong form of the conductivity
equation with soft boundary constraints. Its numerical experiments employ
many boundary excitations and investigate Fourier features and multiscale
boundary inputs. However, the strong form residual contains second
derivatives of the potential and first derivatives of the conductivity. The continuous
least squares loss based on this residual leads to a fourth order
Euler--Lagrange optimality system \cite{CowenBreenEtAl2024}. By contrast, the weak formulation developed
earlier in \cite{BaoEtAl2020} for coefficient inverse problems, including
EIT, avoids second order strong residuals and thereby reduces the order of spatial derivatives required in the loss. 
This reduction comes with a different training structure involving an
additional neural network for the test function and an adversarial
minimax optimization problem.

In this paper, we study two dimensional EIT with Lipschitz conductivities.
We use a scalar stream representation of the electric current to reduce
the conductivity equation to a first order constitutive relation
in strong form. We formulate reconstruction as a direct least squares
minimization over neural trial functions that enforce the conductivity
constraints and prescribed boundary voltages exactly, without introducing
a test network or a minimax problem. The current literature on physics 
informed neural EIT mainly assesses reconstruction accuracy through 
numerical experiments. For our constrained formulation, we establish a rigorous reconstruction error analysis that relates the sampled training loss to reconstruction accuracy. The resulting error estimates account for the contributions from neural approximation, measurement noise, and sampling. Numerical experiments
further demonstrate that the proposed method achieves high reconstruction
accuracy with both exact and noisy boundary data.

The stream representation satisfies current conservation identically. The constitutive residual
contains only first derivatives of the potential and stream function and
no derivative of the conductivity. The conductivity and interior fields
are reconstructed directly from finitely many boundary measurements,
without additional training data or a forward solver during reconstruction.
We impose the conductivity range, its known boundary trace, and the applied
voltages exactly through the neural trial functions. These constraints
hold at every training step and do not require separate terms in the loss
function. It is known that enforcing boundary conditions exactly can improve the
accuracy of physics informed neural solutions
\cite{LuEtAl2021,SukumarSrivastava2022}.
The measured current specifies the tangential derivative of the stream function
on the boundary, and its primitive determines the stream trace up to a
constant. We lift this primitive into the domain to construct a fixed
baseline for the stream network. The stream trace is not imposed as an
exact boundary constraint and can vary during training through the neural
correction. The noisy current is therefore used in both the baseline and
the loss, where its consistency with the reconstructed fields is measured.

To analyze this constrained formulation, we first prove that the neural loss controls the physical boundary current error.
We then prove density results for the constrained neural trial functions representing the conductivity, potentials, and stream functions, showing that the exact constraints preserve the required approximation properties. 
Combining these results with the consistency of the sampled objective gives error estimates with separate contributions from neural approximation, measurement noise, and sampling. 
The boundary current error estimate also yields a conductivity reconstruction estimate under an inverse stability condition for the prescribed finite measurements.

The paper is organized as follows. In Section~\ref{sec:continuous}, we
state the inverse problem and introduce the stream formulation.
Section~\ref{sec:stream-formulation} describes the boundary measurements,
neural trial functions, and minimization problem. We provide the convergence
analysis in Section~\ref{sec:stream-analysis}. Some numerical examples are
presented in Section~\ref{sec:stream-numerics}.
Section~\ref{sec:conclusion} gives concluding remarks and directions for future work.
The appendices contain the technical details.

\section{The inverse problem and stream formulation}
\label{sec:continuous}

In this section, we first state the forward conductivity problem and specify the
admissible conductivities. We then derive the stream function formulation
and define the inverse problem for finitely many boundary measurements.

Let $\Omega\subset\R^2$ be a bounded and simply connected domain. We assume that
$\Omega$ has a $C^{1,1}$ boundary or is a convex curvilinear polygon with
$C^{1,1}$ open edges.
We write $x=(x_1,x_2)$ for points in $\R^2$.
For a positive and bounded conductivity $\sigma$, we consider the electric
potential $u$ satisfying the Dirichlet problem
\begin{equation}
\begin{cases}
\diver(\sigma\nabla u)=0 & \text{in }\Omega,\\
u=f & \text{on }\partial\Omega.
\end{cases}
\label{eq:forward-problem}
\end{equation}
By the Lax--Milgram theorem, this problem has a unique weak solution in
$H^1(\Omega)$ for each $f\in H^{1/2}(\partial\Omega)$.
The boundary measurement corresponding to the applied voltage is the
outgoing current flux. We define the Dirichlet-to-Neumann map from
$H^{1/2}(\partial\Omega)$ to $H^{-1/2}(\partial\Omega)$ in the
standard weak sense by
\begin{equation*}
\Lambda_\sigma f
=\left.\sigma\,\partial_n u_\sigma(f)\right|_{\partial\Omega},
\end{equation*}
where $n=(n_1,n_2)$ is the unit outer normal and $u_\sigma(f)$ is the solution of
\eqref{eq:forward-problem}. The inverse problem is to recover $\sigma$
from $\Lambda_\sigma$. We assume that $\sigma$ is known on the
boundary and reconstruct it in the interior of $\Omega$.

Evaluating the $H^{-1/2}(\partial\Omega)$ norm of the boundary current is
inconvenient numerically. We therefore work with the boundary current as an
$L^2(\partial\Omega)$ function and impose additional regularity on the
conductivity and boundary voltage. We take the conductivity to be Lipschitz continuous
and define the admissible class by
\begin{equation*}
\A=\left\{\sigma\in C^{0,1}(\overline\Omega):
c_1\leq\sigma\leq c_2,\quad
\sigma|_{\partial\Omega}=\sigma_b\right\}
\end{equation*}
where $c_1$ and $c_2$ are fixed constants with $0<c_1\leq c_2$.
The known boundary value $\sigma_b\in C^{0,1}(\partial\Omega)$
satisfies $c_1\leq\sigma_b\leq c_2$.

We use a common notation for the boundary spaces on the two domain classes.
If $\partial\Omega$ is $C^{1,1}$, we set
$\mathcal{H}^s(\partial\Omega)=H^s(\partial\Omega)$. If $\Omega$ is a
convex curvilinear polygon with open edges
$\Gamma_1,\ldots,\Gamma_{N_\Gamma}$, we define
\[
\mathcal{H}^s(\partial\Omega)
:=\left\{v:
v|_{\Gamma_j}\in H^s(\Gamma_j),\ j=1,\ldots,N_\Gamma\right\},
\qquad
\|v\|_{\mathcal{H}^s(\partial\Omega)}^2
:=\sum_{j=1}^{N_\Gamma}\|v\|_{H^s(\Gamma_j)}^2.
\]
In the polygonal case, boundary normal and tangential derivatives are
understood edgewise on the open edges. We also take
$f\in \mathcal{H}^{3/2}(\partial\Omega)\cap C(\partial\Omega)$.
For a $C^{1,1}$ boundary,
$\mathcal{H}^{3/2}(\partial\Omega)=H^{3/2}(\partial\Omega)
\hookrightarrow C(\partial\Omega)$, and therefore
\[
\mathcal{H}^{3/2}(\partial\Omega)\cap C(\partial\Omega)
=H^{3/2}(\partial\Omega)
=\mathcal{H}^{3/2}(\partial\Omega).
\]
For a convex curvilinear polygon, continuity ensures compatibility of the
boundary values at the corners.
By elliptic regularity on the two domain classes
\cite{GilbargTrudinger2001,Grisvard1985}, we have
$u_\sigma(f)\in H^2(\Omega)$. In particular,
\begin{equation*}
\sigma\, \partial_n u_\sigma(f)
\in \mathcal{H}^{1/2}(\partial\Omega)
\hookrightarrow L^{2}(\partial\Omega).
\end{equation*}

The following proposition gives an equivalent first order formulation
of the conductivity equation in terms of a scalar stream function.
\begin{proposition}
\label{prop:stream-equivalence}
Let $\sigma\in \A$ and
$u\in H^2(\Omega)$.  Then
\begin{equation}
\diver(\sigma\nabla u)=0
\label{eq:conductivity-equation}
\end{equation}
if and only if there is a function $\psi\in H^2(\Omega)$ such that
\begin{equation}
\nabla^\perp\psi=\sigma\nabla u,
\quad\mbox{where}\quad
\nabla^\perp\psi=(\partial_{x_2}\psi,-\partial_{x_1}\psi)^T.
\label{eq:stream-relation}
\end{equation}
Such a stream function $\psi$ is uniquely determined up to an additive constant.
\end{proposition}

\begin{proof}
If \eqref{eq:conductivity-equation} holds, then
$\sigma\nabla u\in H^1(\Omega;\R^2)$ is divergence free.
The planar potential theorem on a simply connected domain gives
$\sigma\nabla u=\nabla^\perp\psi$ for some $\psi\in H^2(\Omega)$,
unique up to an additive constant. Conversely,
$\diver(\nabla^\perp\psi)=0$ almost everywhere in $\Omega$,
which proves \eqref{eq:conductivity-equation}.
\end{proof}

The proposition gives a stream function in $H^2(\Omega)$.
Since $H^2(\Omega)\hookrightarrow C(\overline{\Omega})$, its point
values are well defined. Adding a constant to $\psi$ leaves the current
$\nabla^\perp\psi$ unchanged. We fix this constant by choosing a
reference point $x_\star\in\partial\Omega$ and requiring the stream
function to vanish at $x_\star$. For $\sigma\in\A$ and
$f\in \mathcal{H}^{3/2}(\partial\Omega)\cap C(\partial\Omega)$,
we denote by $\psi_\sigma(f)\in H^2(\Omega)$ the unique function satisfying
\begin{equation}
\nabla^\perp\psi_\sigma(f)
=\sigma\nabla u_\sigma(f),
\qquad
\psi_\sigma(f)(x_\star)=0.
\label{eq:forward-stream}
\end{equation}
We call $u_\sigma(f)$ and $\psi_\sigma(f)$ the exact forward fields associated with the given conductivity $\sigma$ and boundary voltage $f$. An exact forward tuple consists of an admissible conductivity and its corresponding fields for the prescribed boundary voltages.

The stream formulation enforces current conservation identically.
Its constitutive equation \eqref{eq:stream-relation} involves first
derivatives of the potential and stream function, but no derivative of
the conductivity. It also relates the stream trace directly to the
boundary current. In fact, letting $\tau=(-n_2,n_1)^T$ be the positively
oriented unit tangent, we then have
\begin{equation}
\partial_\tau\psi_\sigma(f)
=(\nabla^\perp\psi_\sigma(f))\cdot n
=\sigma_b\,\partial_n u_\sigma(f)
=\Lambda_\sigma f.
\label{eq:forward-current}
\end{equation}
Thus, fixing the value at $x_\star$ determines the stream trace by
integration of the current along $\partial\Omega$. Let
$L=|\partial\Omega|$ and parameterize the boundary by arc length
$s\in[0,L]$, starting at $x_\star$ in the positive orientation.
Equations~\eqref{eq:forward-stream} and \eqref{eq:forward-current} give
\begin{equation}
(\Tr\psi_\sigma(f))(s)
=
\int_0^s (\Lambda_\sigma f)(r)\,dr,
\qquad 0\leq s\leq L,
\label{eq:boundary-stream-primitive}
\end{equation}
where $Tr$ denotes the trace operator.
At $s=L$, the boundary returns to $x_\star$. We therefore obtain the
compatibility condition
\begin{equation}
\int_{\partial\Omega}\Lambda_\sigma f\,ds=0,
\label{eq:current-compatibility}
\end{equation}
which also expresses conservation of current.

We now restrict the inverse problem to finitely many measurements.
Let $M$ be the number of boundary experiments and let
$f_i\in \mathcal{H}^{3/2}(\partial\Omega)\cap C(\partial\Omega)$,
$i=1,\ldots,M$, be the applied boundary voltages. Fix a true conductivity
$\sigma^{\mathrm{true}}\in\A$ and let
\begin{equation*}
g_i=\Lambda_{\sigma^{\mathrm{true}}}f_i
=\partial_\tau\psi_{\sigma^{\mathrm{true}}}(f_i),
\qquad i=1,\ldots,M,
\end{equation*}
be the exact currents. The preceding regularity result gives
$g_i\in \mathcal{H}^{1/2}(\partial\Omega)$. We seek to recover
$\sigma^{\mathrm{true}}$ from the voltage and current pairs
$\{(f_i,g_i)\}_{i=1}^M$.

\section{The neural minimization problem}
\label{sec:stream-formulation}

In this section, we formulate a constrained neural minimization problem
for EIT with finite boundary measurements. We use neural networks to represent
the conductivity, potentials, and stream functions. The trial functions
satisfy the prescribed constraints exactly. The loss measures the
constitutive residual and the discrepancy between the measured boundary
currents and those computed from the neural stream functions.

Using the stream function introduced in Section~\ref{sec:continuous}, we
construct the neural formulation based on the following first order system.
\begin{equation*}
\begin{aligned}
&\sigma\in\A,\qquad
\nabla^\perp\psi_i=\sigma\nabla u_i
\quad\text{in }\Omega,\\
&\Tr u_i=f_i,\quad \partial_\tau\psi_i=g_i
\quad\text{on }\partial\Omega,\qquad
\psi_i(x_\star)=0,\quad i=1,\ldots,M.
\end{aligned}
\end{equation*}
We also consider the measured currents $g_i^\delta\in\mathcal{H}^{1/2}(\partial\Omega)$, 
which may contain noise.

\subsection{Neural representation of the unknown fields}
\label{sec:neural-architectures}

One conductivity network with a single output is shared by all experiments.
For each experiment, a separate field network with two outputs is used to
construct the potential and stream function. All networks take the spatial
coordinates $x\in\overline\Omega\subset\mathbb R^2$ as input.

Set $m_0=2$. Let
$m_1,\ldots,m_{L-1}$ denote the hidden layer widths and $m_L$ the
output dimension. We define a fully connected network with $L-1$ hidden
layers from $\overline\Omega$ to $\mathbb R^{m_L}$ by
\begin{equation*}
\begin{aligned}
h_0(x)&=x,\\
h_\ell(x)&=\varphi_{\mathrm{act}}
\bigl(W_\ell h_{\ell-1}(x)+b_\ell\bigr),
&&\ell=1,\ldots,L-1,\\
h_{L}(x)&=W_Lh_{L-1}(x)+b_L.
\end{aligned}
\end{equation*}
Here
$W_\ell\in\mathbb R^{m_\ell\times m_{\ell-1}}$ and
$b_\ell\in\mathbb R^{m_\ell}$ ($\ell=1,\ldots,L$) are the weights and biases,
and the activation $\varphi_{\mathrm{act}}$ is applied componentwise.
For the conductivity and the fields, we write
\begin{equation}
\begin{aligned}
N_\sigma^\vartheta&:\overline\Omega\to\mathbb R,\\
\Phi_i^{\theta_i}&:\overline\Omega\to\mathbb R^2,
&\Phi_i^{\theta_i}(x)
&=\bigl(U_i^{\theta_i}(x),P_i^{\theta_i}(x)\bigr),
&&i=1,\ldots,M.
\end{aligned}
\label{eq:neural-unknowns}
\end{equation}
The vector $\vartheta$ contains the conductivity network parameters.
The vector $\theta_i$ contains the field network parameters for experiment $i$. The outputs
$U_i^{\theta_i}$ and $P_i^{\theta_i}$ are used to construct the potential and
stream function, respectively. The conductivity and field
networks need not have the same depth or width. Once the architectures are
fixed, we collect their parameters in
\begin{equation*}
\Theta=(\vartheta,\theta_1,\ldots,\theta_M)
\in\mathbb R^{N_{\mathrm{par}}},
\end{equation*}
where $N_{\mathrm{par}}$ is the total number of scalar weights and biases.

The raw network outputs in \eqref{eq:neural-unknowns} are unconstrained. In the next
subsection, we use them to construct trial functions with the required
physical and boundary conditions. To ensure the regularity needed for the trial functions and establish the 
approximation results in Section~\ref{sec:neural-approximation}, we assume
\begin{equation}
\varphi_{\mathrm{act}}\in C^2(\mathbb R),
\qquad
\varphi_{\mathrm{act}}\ \text{is bounded and nonconstant}.
\label{eq:activation-assumption}
\end{equation}
We use the $\tanh$ activation in the numerical experiments.

\subsection{Trial functions with exact constraints}
\label{sec:neural-hard-constraints}

We next construct the trial functions. To impose the conductivity bounds,
we introduce the clipping map
\begin{equation*}
\clip_{[c_1,c_2]}(t)=\max\{c_1,\min\{t,c_2\}\}.
\end{equation*}
The clipping map is Lipschitz.
To impose the boundary values, we choose a boundary factor function satisfying
\begin{equation}
\rho_\Omega=0\ \text{on }\partial\Omega,
\qquad
\rho_\Omega>0\ \text{in }\Omega, \qquad \rho_\Omega\in C^{1,1}(\overline\Omega).
\label{eq:boundary-factor-general}
\end{equation}
These properties allow us to enforce the boundary constraints and prove
the density result in the next section. A construction of the
boundary factor is given in Appendix~\ref{app:boundary-factors}.

We begin with the admissible conductivity class $\mathcal A$.
Choose a fixed extension $A_b\in C^{0,1}(\overline\Omega)$ of the known
boundary conductivity such that $A_b|_{\partial\Omega}=\sigma_b$.
We define the conductivity trial function by
\begin{equation}
\sigma^\vartheta(x)
=\clip_{[c_1,c_2]}\!\left(A_b(x)+\rho_\Omega(x)N_\sigma^\vartheta(x)\right).
\label{eq:conductivity-hard-map}
\end{equation}
For every $\vartheta$, the clipping map gives $c_1\leq\sigma^\vartheta\leq c_2$,
and the vanishing boundary factor gives $\sigma^\vartheta|_{\partial\Omega}=\sigma_b$.
The function $A_b(x)+\rho_\Omega(x)N_\sigma^\vartheta(x)$ is Lipschitz.
Its composition with the clipping map is therefore Lipschitz as well, so
$\sigma^\vartheta\in\mathcal A$.

For the potential, we impose the hard constraint on the boundary values.
For each applied voltage $f_i$ ($i=1,\ldots,M$), choose a fixed lifting
$F_i \in H^2(\Omega)$ satisfying $\Tr F_i=f_i$. We define the potential by
\begin{equation}
u_i^{\theta_i}(x)=F_i(x)+\rho_\Omega(x)U_i^{\theta_i}(x).
\label{eq:field-hard-representations}
\end{equation}
This construction gives $u_i^{\theta_i}\in H^2(\Omega)$ and enforces
$\Tr u_i^{\theta_i}=f_i$ for every parameter value.

We treat the stream function differently. For each experiment, the measured current
$g_i^\delta\in\mathcal{H}^{1/2}(\partial\Omega)$ may contain noise, so we do not impose 
a hard constraint on the
stream boundary values. Instead, we use these data to construct a baseline,
and the stream trace remains free. The measurements need not satisfy
\eqref{eq:current-compatibility}. For the baseline construction only, choose
a correction $r_i^\delta\in \mathcal{H}^{1/2}(\partial\Omega)$ such that
\begin{equation*}
\int_{\partial\Omega}r_i^\delta\,ds
=-\int_{\partial\Omega}g_i^\delta\,ds,
\end{equation*}
and set $\widetilde g_i^\delta=g_i^\delta+r_i^\delta$. The corrected data satisfy
$\int_{\partial\Omega}\widetilde g_i^\delta\,ds=0$.
We use the reference point $x_\star\in\partial\Omega$ from
Section~\ref{sec:continuous} and the positively oriented arclength parameter
$s\in[0,L]$ in \eqref{eq:boundary-stream-primitive}. The boundary stream primitive is
\begin{equation*}
\zeta_i^\delta(s)=\int_0^s \widetilde g_i^\delta(r)\,dr,
\qquad 0\le s\le L.
\end{equation*}
Since $\widetilde g_i^\delta\in \mathcal{H}^{1/2}(\partial\Omega)$ has zero
boundary integral, $\zeta_i^\delta$ has one value at each point of the closed
boundary and belongs to $\mathcal{H}^{3/2}(\partial\Omega)\cap C(\partial\Omega)$.
For either domain class considered in this paper, the trace theorem provides an
$H^2(\Omega)$ lifting. We choose a baseline $B_i\in H^2(\Omega)$ with
$\Tr B_i=\zeta_i^\delta$ and $B_i(x_\star)=0$, and keep it fixed throughout
training. We define the stream function by
\begin{equation}
\psi_i^{\theta_i}(x)
=B_i(x)+P_i^{\theta_i}(x)-P_i^{\theta_i}(x_\star).
\label{eq:stream-network-representation}
\end{equation}
Subtracting $P_i^{\theta_i}(x_\star)$ fixes the additive constant.
Since $P_i^{\theta_i}$ is free on $\partial\Omega$, the baseline
does not impose the measured current as an exact boundary condition.

The resulting trial functions then satisfy
\begin{equation}
\sigma^\vartheta\in\mathcal A, \qquad
\Tr u_i^{\theta_i}=f_i,
\qquad
\diver(\nabla^\perp\psi_i^{\theta_i})=0
\quad\text{a.e. in }\Omega,
\qquad
\psi_i^{\theta_i}(x_\star)=0
\label{eq:hard-constraint-summary}
\end{equation}
for every parameter value. We therefore fit only the constitutive relation
and the measured normal current.

\subsection{The continuous minimization problem}
\label{sec:continuous-neural-loss}

We collect the trial functions in
\begin{equation}
W^\Theta
=\bigl(\sigma^\vartheta,
 u_1^{\theta_1},\psi_1^{\theta_1},\ldots,
 u_M^{\theta_M},\psi_M^{\theta_M}\bigr),
\label{eq:neural-physical-tuple}
\end{equation}
and define the constitutive and data residuals for each experiment by
\begin{equation*}
s_i(W^\Theta)
=\nabla^\perp\psi_i^{\theta_i}
-\sigma^\vartheta\nabla u_i^{\theta_i},
\qquad
d_i^\delta(W^\Theta)
=\partial_\tau\psi_i^{\theta_i}-g_i^\delta.
\end{equation*}
Define the three continuous loss components by
\[
\begin{aligned}
\mathcal L_\Omega(W^\Theta)
&=\frac1M\sum_{i=1}^M\|s_i(W^\Theta)\|_{L^2(\Omega)}^2,\\
\mathcal L_b(W^\Theta)
&=\frac1M\sum_{i=1}^M\|\Tr s_i(W^\Theta)\|_{L^2(\partial\Omega)}^2,\\
\mathcal L_q^\delta(W^\Theta)
&=\frac1M\sum_{i=1}^M\|d_i^\delta(W^\Theta)\|_{L^2(\partial\Omega)}^2.
\end{aligned}
\]
For fixed $\lambda_b,\lambda_q>0$, define the continuous loss by
\begin{equation}
\mathcal L^\delta(W^\Theta)
=\mathcal L_\Omega(W^\Theta)
+\lambda_b\, \mathcal L_b(W^\Theta)
+\lambda_q\, \mathcal L_q^\delta(W^\Theta),
\label{eq:stream-residual-loss}
\end{equation}
and express it in terms of the neural parameters as
\begin{equation*}
\J^\delta(\Theta)=\mathcal L^\delta(W^\Theta).
\end{equation*}
For exact data, we write $\mathcal L=\mathcal L^0$ and $\J=\J^0$.
The three terms in \eqref{eq:stream-residual-loss} measure the interior
constitutive residual, its boundary trace, and the discrepancy between the
measured and computed boundary currents, respectively.

To help stabilize the reconstruction, we use the exact constraints in \eqref{eq:hard-constraint-summary} 
and the boundary penalty $\mathcal L_b(W^\Theta)$, without adding other regularization terms 
to the loss in \eqref{eq:stream-residual-loss}. The constraints already incorporate prior information directly into the trial functions. For neural approximations, we measure the interior
constitutive residual in the $L^2$ norm. 
Using the $H^1$ norm in the loss would require evaluating derivatives of the conductivity and second derivatives of the potential and stream function. However, smallness of
$\|s_i(W^\Theta)\|_{L^2(\Omega)}$ alone does not imply smallness of its boundary
trace. The term $\mathcal L_b(W^\Theta)$ directly penalizes this trace in
$L^2(\partial\Omega)$ without increasing the differential order of the loss.
The boundary term is also well defined under the assumed regularity. In fact,
$u_i^{\theta_i},\psi_i^{\theta_i}\in H^2(\Omega)$ and
$\sigma^\vartheta\in C^{0,1}(\overline\Omega)$ imply
\[
s_i(W^\Theta)\in H^1(\Omega;\R^2),
\qquad
\Tr s_i(W^\Theta)\in H^{1/2}(\partial\Omega;\R^2)
\hookrightarrow L^2(\partial\Omega;\R^2).
\]
For an exact forward tuple, \eqref{eq:forward-stream} gives $s_i=0$ in $\Omega$, 
so its trace also vanishes. Thus the boundary term is consistent with the 
governing equations and imposes no additional condition on an exact solution.

Our neural formulation is summarized in Figure~\ref{fig:nn-method}. 

\begin{figure}[!htbp]
\centering
\IfFileExists{figures/nn_method.pdf}{%
\includegraphics[width=\linewidth,height=0.375\textheight,keepaspectratio]{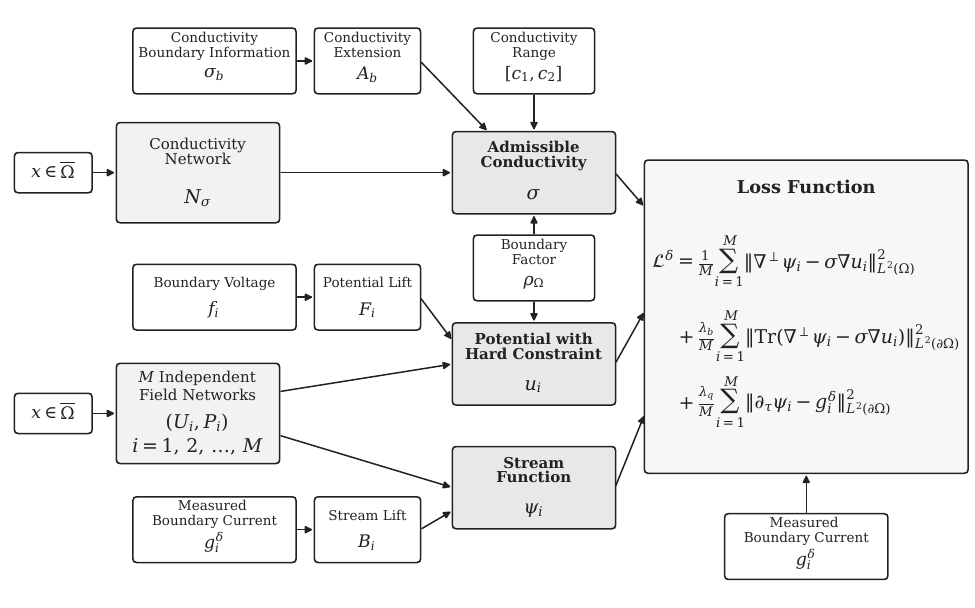}}{%
\fbox{\parbox[c][0.345\textheight][c]{0.94\linewidth}{\centering
Neural method diagram\\[6pt]
\texttt{figures/nn\_method.pdf}}}}
\caption{Neural stream function method for EIT. The constrained neural
trial functions are trained using the constitutive residuals and measured
boundary currents.}
\label{fig:nn-method}
\end{figure}

Now we state the minimization problem. For fixed neural network architectures
and $R>0$, let
\begin{equation}
\mathcal K_R=[-R,R]^{N_{\mathrm{par}}}
\label{eq:parameter-box}
\end{equation}
be the compact parameter set. We consider the minimization problem
\begin{equation}
\min_{\Theta\in\mathcal K_R}\J^\delta(\Theta).
\label{eq:fixed-architecture-minimum}
\end{equation}

\begin{proposition}
\label{prop:fixed-architecture-minimum}
For fixed finite neural network architectures and $R>0$, $\J^\delta$ is continuous on
$\mathcal K_R$. The minimization problem
\eqref{eq:fixed-architecture-minimum} then has a minimizer.
\end{proposition}

\begin{proof}
For fixed finite architectures, the raw network outputs depend continuously 
on $\Theta$ in $C^2(\overline{\Omega})$. The trial function constructions, the Lipschitz clipping map, 
and the trace theorem therefore imply that every term in $\J^\delta$ is continuous in $\Theta$. 
Thus $\J^\delta$ is continuous on the compact set $\mathcal K_R$ and attains its minimum.
\end{proof}

\section{Convergence analysis}
\label{sec:stream-analysis}

In this section, we analyze the conductivity reconstructed by the neural
minimization problem. We define the physical current residual and
the set of conductivities that reproduce the exact finite measurements.
\begin{equation*}
\begin{aligned}
\mathcal R(\sigma)
&:=\frac1M\sum_{i=1}^M
\|\Lambda_\sigma f_i-g_i\|_{H^{-1/2}(\partial\Omega)}^2,\\
\mathcal M
&:=\{\sigma\in\A:\mathcal R(\sigma)=0\}.
\end{aligned}
\end{equation*}
The set $\mathcal M$ is nonempty since
$\sigma^{\mathrm{true}}\in\mathcal M$. Both
$\Lambda_\sigma f_i$ and $g_i$ are in $L^2(\partial\Omega)$.
Their equality in $H^{-1/2}(\partial\Omega)$ is therefore equivalent to
equality almost everywhere on the boundary. Consequently,
\begin{equation*}
\mathcal M
=\{\sigma\in\A:\mathcal R(\sigma)=0\}
=\{\sigma\in\A:\Lambda_\sigma f_i=g_i
\text{ in }L^2(\partial\Omega),\
i=1,\ldots,M\}.
\end{equation*}
We first prove the consistency of the continuous loss and bound $\mathcal R$ in terms of the loss and measurement noise. We also obtain a conditional estimate of the conductivity’s distance to $\mathcal M$. We then establish the approximation properties of the constrained trial functions. Finally, we derive error estimates and convergence results for the sampled problem, accounting for neural approximation, measurement noise, and sampling.

\subsection{Consistency of the continuous problem}
\label{sec:zero-loss-consistency}

Let $W^\Theta$ be a trial tuple as in \eqref{eq:neural-physical-tuple}, with
\begin{equation}
\sigma^\vartheta\in\A,\qquad
u_i^{\theta_i},\psi_i^{\theta_i}\in H^2(\Omega),\qquad
\Tr u_i^{\theta_i}=f_i,\qquad
\psi_i^{\theta_i}(x_\star)=0,
\quad i=1,\ldots,M.
\label{eq:joint-state-regularity}
\end{equation}
For the conductivity $\sigma^\vartheta$, let
$u_{\sigma^\vartheta}(f_i)$ and $\psi_{\sigma^\vartheta}(f_i)$ denote the
exact forward fields. 
We show that zero loss characterizes the exact inverse solutions and their
corresponding forward fields.

\begin{proposition}
\label{prop:zero-loss-characterization}
For exact data and every trial tuple $W^\Theta$ satisfying
\eqref{eq:joint-state-regularity}, we have
\begin{equation*}
\mathcal L(W^\Theta)=0
\quad\Longleftrightarrow\quad
\sigma^\vartheta\in\mathcal M,\qquad
u_i^{\theta_i}=u_{\sigma^\vartheta}(f_i),\qquad
\psi_i^{\theta_i}=\psi_{\sigma^\vartheta}(f_i)
\quad\text{for all }i.
\end{equation*}
\end{proposition}

\begin{proof}
Suppose first that $\mathcal L(W^\Theta)=0$. All terms in $\mathcal L$ are nonnegative,
so each term vanishes. In particular, $s_i(W^\Theta)=0$ in $L^2(\Omega;\R^2)$ gives
\[
\nabla^\perp\psi_i^{\theta_i}
=\sigma^\vartheta\nabla u_i^{\theta_i}
\quad\text{a.e. in }\Omega.
\]
Since $\diver(\nabla^\perp\psi_i^{\theta_i})=0$ a.e. in $\Omega$, we obtain
\[
\diver(\sigma^\vartheta\nabla u_i^{\theta_i})=0
\quad\text{a.e. in }\Omega.
\]
The condition $\Tr u_i^{\theta_i}=f_i$ and uniqueness of the Dirichlet
problem \eqref{eq:forward-problem} imply
$u_i^{\theta_i}=u_{\sigma^\vartheta}(f_i)$. The data residual
$d_i^0(W^\Theta)=0$ in $L^2(\partial\Omega)$ gives
$\partial_\tau\psi_i^{\theta_i}=g_i$ in $L^2(\partial\Omega)$.
Combining this identity with $\Tr s_i(W^\Theta)=0$ in $L^2(\partial\Omega)$,
we obtain
\[
\Lambda_{\sigma^\vartheta}f_i = \sigma^\vartheta \partial_n u_{\sigma^\vartheta}(f_i)
=\sigma^\vartheta\partial_n u_i^{\theta_i}
=(\nabla^\perp\psi_i^{\theta_i})\cdot n
=\partial_\tau\psi_i^{\theta_i}
=g_i.
\]
Thus $\sigma^\vartheta\in\mathcal M$. Finally,
\[
\nabla^\perp\psi_i^{\theta_i}=\sigma^\vartheta\nabla u_i^{\theta_i} 
=\sigma^\vartheta\nabla u_{\sigma^\vartheta}(f_i)
=\nabla^\perp\psi_{\sigma^\vartheta}(f_i).
\]
It follows that $\psi_i^{\theta_i}-\psi_{\sigma^\vartheta}(f_i)$ is constant
in $\Omega$. Both functions vanish at $x_\star$, so this constant is zero.

Conversely, suppose that $\sigma^\vartheta\in\mathcal M$ and $W^\Theta$ is the 
corresponding exact forward tuple. 
Then $s_i(W^\Theta)=0$ and $d_i^0(W^\Theta)=0$ for every $i$. Since
$s_i(W^\Theta)\in H^1(\Omega;\R^2)$, its trace vanishes as well.
Every term in $\mathcal L(W^\Theta)$ is therefore zero.
\end{proof}

\subsection{Control of the physical residual}
\label{sec:residual-forward-control}

We next estimate the boundary current residual in terms of the neural loss
and the measurement noise. Suppose that the noisy currents
$g_i^\delta\in \mathcal{H}^{1/2}(\partial\Omega)$ satisfy
\begin{equation}
\frac1M\sum_{i=1}^M
\|g_i^\delta-g_i\|_{L^2(\partial\Omega)}^2
\leq\delta^2.
\label{eq:noise-model}
\end{equation}
The noise bound in $L^2(\partial\Omega)$ allows us to compare the loss for
$g_i^\delta$ with the loss for the exact data $g_i$. The physical residual
$\mathcal R$ measures the boundary current in its natural weak space
$H^{-1/2}(\partial\Omega)$. The embedding
$L^2(\partial\Omega)\hookrightarrow H^{-1/2}(\partial\Omega)$ allows the
boundary current term in the neural loss to help control $\mathcal R$.

\begin{theorem}
\label{lem:stream-current}
Let $W^\Theta$ be any trial tuple satisfying
\eqref{eq:joint-state-regularity}. Then
\begin{equation}
\mathcal R(\sigma^\vartheta)
\leq C\left[\mathcal L^\delta(W^\Theta)+\delta^2\right],
\label{eq:stream-residual-current}
\end{equation}
where $C$ depends on $\Omega$, $c_1$, $c_2$, and $\lambda_q$. 
\end{theorem}

\begin{proof}
Set $\sigma=\sigma^\vartheta$, $u_i=u_i^{\theta_i}$, and
$\psi_i=\psi_i^{\theta_i}$. For each $i$, we decompose the current error as
\begin{align*}
\Lambda_\sigma f_i-g_i
&=(\Lambda_\sigma f_i-\partial_\tau\psi_i)
+(\partial_\tau\psi_i-g_i^\delta)
+(g_i^\delta-g_i).
\end{align*}
The last two terms belong to $L^2(\partial\Omega)$, which embeds continuously
into $H^{-1/2}(\partial\Omega)$. We have
\begin{align}
\|\Lambda_\sigma &{}f_i-g_i\|_{H^{-1/2}(\partial\Omega)}
\leq
\|\Lambda_\sigma f_i-\partial_\tau\psi_i\|_{H^{-1/2}(\partial\Omega)}
+C\|\partial_\tau\psi_i-g_i^\delta\|_{L^2(\partial\Omega)}
+C\|g_i^\delta-g_i\|_{L^2(\partial\Omega)}\notag\\
&\leq C(
\|\Lambda_\sigma f_i-\partial_\tau\psi_i\|_{H^{-1/2}(\partial\Omega)}
+\|d_i^\delta(W^\Theta)\|_{L^2(\partial\Omega)}
+\|g_i^\delta-g_i\|_{L^2(\partial\Omega)}
).
\label{eq:initial-current-error-decomposition}
\end{align}
We next bound the first term on the right. Define
\[
K_i=\sigma\nabla u_\sigma(f_i)-\nabla^\perp\psi_i.
\]
Both fluxes are divergence free, so $\diver K_i=0$. 
Its normal trace is
\[
K_i\cdot n=\Lambda_\sigma f_i-\partial_\tau\psi_i
\qquad\text{on }\partial\Omega.
\]
The normal trace theorem then gives
\[
\|\Lambda_\sigma f_i-\partial_\tau\psi_i\|_{H^{-1/2}(\partial\Omega)}=
\|K_i\cdot n\|_{H^{-1/2}(\partial\Omega)}
\leq C\|K_i\|_{L^2(\Omega)}.
\]
To bound $K_i$, let $e_i=u_\sigma(f_i)-u_i\in H_0^1(\Omega)$, where the 
zero boundary trace follows because both potentials have trace $f_i$.
From the definition
$s_i(W^\Theta)=\nabla^\perp\psi_i-\sigma\nabla u_i$ we know
\begin{equation}
K_i=\sigma\nabla e_i -s_i(W^\Theta).
\label{eq:Ki-residual-error}
\end{equation}
Using $\diver K_i=0$, for every $\varphi\in H_0^1(\Omega)$, we obtain
\[
0=\int_\Omega K_i\cdot\nabla\varphi\,dx
=\int_\Omega\sigma\nabla e_i\cdot\nabla\varphi\,dx-\int_\Omega s_i(W^\Theta)\cdot\nabla\varphi\,dx.
\]
Taking $\varphi=e_i$ and using $\sigma\geq c_1$, we have
\begin{equation}
\|\nabla e_i\|_{L^2(\Omega)}
\leq c_1^{-1}\|s_i(W^\Theta)\|_{L^2(\Omega)}.
\label{eq:stream-energy-bound}
\end{equation}
Using \eqref{eq:Ki-residual-error}, $\sigma\leq c_2$, and
\eqref{eq:stream-energy-bound}, we know
\begin{align*}
\|K_i\|_{L^2(\Omega)}
&\leq c_2\|\nabla e_i\|_{L^2(\Omega)} + \|s_i(W^\Theta)\|_{L^2(\Omega)}
\leq\big(1+\frac{c_2}{c_1}\big)
\|s_i(W^\Theta)\|_{L^2(\Omega)}.
\end{align*}
Combining the preceding estimates with the initial current error
estimate \eqref{eq:initial-current-error-decomposition}, we get
\begin{align*}
\|\Lambda_\sigma f_i-g_i\|_{H^{-1/2}(\partial\Omega)}
&\leq C\bigl(
\|s_i(W^\Theta)\|_{L^2(\Omega)}
+\|d_i^\delta(W^\Theta)\|_{L^2(\partial\Omega)}
+\|g_i^\delta-g_i\|_{L^2(\partial\Omega)}
\bigr).
\end{align*}
Squaring this estimate gives
\[
\|\Lambda_\sigma f_i-g_i\|_{H^{-1/2}(\partial\Omega)}^2
\leq C\bigl(
\|s_i(W^\Theta)\|_{L^2(\Omega)}^2
+\|d_i^\delta(W^\Theta)\|_{L^2(\partial\Omega)}^2
+\|g_i^\delta-g_i\|_{L^2(\partial\Omega)}^2
\bigr).
\]
Averaging over $i=1,\ldots,M$ and using the noise bound \eqref{eq:noise-model}, we obtain
\[
\mathcal R(\sigma)
\leq \frac{C}{M}\sum_{i=1}^{M}
\left(
\|s_i(W^\Theta)\|_{L^2(\Omega)}^2
+\|d_i^\delta(W^\Theta)\|_{L^2(\partial\Omega)}^2
\right)+C\delta^2
\leq C(\mathcal L^\delta(W^\Theta)+\delta^2),
\]
where the last inequality follows from the definition of $\mathcal L^\delta(W^\Theta)$
in \eqref{eq:stream-residual-loss}. 
\end{proof}

To relate the residual to the conductivity error, define
\begin{equation*}
\operatorname{dist}_{C(\overline\Omega)}(\sigma,\mathcal M)
:=\inf_{\widetilde\sigma\in\mathcal M}
\|\sigma-\widetilde\sigma\|_{C(\overline\Omega)}.
\end{equation*}
We use the following condition.
\begin{assumption}
\label{ass:solution-set-stability}
There is a nondecreasing function $\omega:[0,\infty)\to[0,\infty)$ such that
$\omega(t)\to 0$ as $t\to 0$ and $\operatorname{dist}_{C(\overline\Omega)}(\sigma,\mathcal M)
\leq\omega (\mathcal R(\sigma))$ for every $\sigma\in\A$.
\end{assumption}
This condition means that a small residual for the prescribed finite measurements gives a small uniform distance to 
$\mathcal M$. For the inverse problem with the full Dirichlet-to-Neumann map, 
a corresponding stability estimate holds under uniform a priori H\"older bounds for the conductivity \cite{BarceloFaracoRuiz2007}. For the inverse problem with finitely many measurements,
stability depends on several factors, including the admissible conductivities,
the number and choice of boundary data, and the independent and complementary
information provided by these data. Existing uniqueness and stability results for EIT with sufficiently informative finite measurements and suitable a priori restrictions on the conductivity support this type of stability 
assumption \cite{AlbertiSantacesaria2022,Harrach2019}.
Conditional stability estimates have also been used to bound PINN solution errors in terms of training and quadrature errors for data assimilation problems \cite{MishraMolinaro2022}.

\begin{corollary}
\label{cor:stream-conductivity}
Let $W^\Theta$ be any trial tuple satisfying
\eqref{eq:joint-state-regularity}.
If Assumption~\ref{ass:solution-set-stability}
holds, then
\begin{equation}
\operatorname{dist}_{C(\overline\Omega)}(\sigma^\vartheta,\mathcal M)
\leq
\omega\!\left(C[\mathcal L^\delta(W^\Theta)+\delta^2]\right),
\label{eq:joint-loss-to-solution-set}
\end{equation}
where $C$ is the constant in Theorem~\ref{lem:stream-current}.
\end{corollary}

\begin{proof}
The result follows from Assumption~\ref{ass:solution-set-stability} and
\eqref{eq:stream-residual-current}.
\end{proof}

\subsection{Approximation by the constrained neural trial functions}
\label{sec:neural-approximation}

The trial fields in Section~\ref{sec:neural-hard-constraints} satisfy the
prescribed hard constraints for every parameter value. We now show that these
constraints do not reduce the approximation ability of the trial class.
The usual uniform approximation theorem does not apply directly because the
approximating functions must also satisfy the constraints. We first establish
the required density properties.

\begin{lemma}
\label{lem:rho-density}
The boundary factor $\rho_\Omega$ in
\eqref{eq:boundary-factor-general} can be chosen so that
\begin{equation}
\overline{\rho_\Omega C(\overline\Omega)}^{\,C(\overline\Omega)}
=\{h\in C(\overline\Omega):h|_{\partial\Omega}=0\}
\label{eq:rho-density}
\end{equation}
and
\begin{equation}
\overline{\rho_\Omega C^2(\overline\Omega)}^{\,H^2(\Omega)}
=H^2(\Omega)\cap H_0^1(\Omega).
\label{eq:rho-H2-density}
\end{equation}
\end{lemma}

Related completeness results for trial functions incorporating
boundary factors have been established under stronger smoothness
assumptions \cite{RvachevEtAl2000}.
We use \eqref{eq:rho-density} to approximate the conductivity correction
$\sigma-A_b$ uniformly and \eqref{eq:rho-H2-density} to approximate the
potential correction $u_\sigma(f_i)-F_i$ in $H^2(\Omega)$.
These properties yield the constrained approximations in
Proposition~\ref{prop:constrained-field-approximation}.
The $H^2$ density assertion \eqref{eq:rho-H2-density} requires particular care. 
Although every function of the form $\rho_\Omega q$ has zero boundary trace, 
this alone does not establish the asserted density.
A function in $H^2(\Omega)\cap H_0^1(\Omega)$ may have a nonzero normal trace and hence cannot,
in general, be approximated in $H^2(\Omega)$ by compactly supported functions.
Direct division by $\rho_\Omega$ may also be singular near $\partial\Omega$.
The construction of
$\rho_\Omega$ and the proofs of both identities are given in
Appendix~\ref{app:boundary-factors}.

We also use the following standard neural approximation results for functions
and their derivatives \cite{Hornik1991}.
\begin{lemma}
\label{thm:neural-approximation}
Suppose that $\varphi_{\mathrm{act}}$ satisfies
\eqref{eq:activation-assumption}. Then, for every $\varepsilon>0$,
\begin{enumerate}
\item for every $q\in C(\overline\Omega)$, there exists a finite scalar neural
network whose output $N$ satisfies
\[
\|N-q\|_{C(\overline\Omega)}<\varepsilon;
\]
\item for every $q\in C^2(\overline\Omega)$, there exists a finite scalar neural
network whose output $N$ satisfies
\[
\max_{|\alpha|\leq2}
\|D^\alpha N-D^\alpha q\|_{C(\overline\Omega)}<\varepsilon.
\]
\end{enumerate}
\end{lemma}

We now state our approximation results for the constrained neural class.

\begin{proposition}
\label{prop:constrained-field-approximation}
Let $\sigma\in\A$, and let $B_i\in H^2(\Omega)$, $i=1,\ldots,M$, be
any fixed stream baselines satisfying $B_i(x_\star)=0$. For every
$\varepsilon>0$, there are finite neural networks with the activation
$\varphi_{\mathrm{act}}$ satisfying \eqref{eq:activation-assumption} such that
\begin{equation}
\|\sigma^\vartheta-\sigma\|_{C(\overline\Omega)}^2
+\frac1M\sum_{i=1}^M\left(
\|u_i^{\theta_i}-u_\sigma(f_i)\|_{H^2(\Omega)}^2
+\|\psi_i^{\theta_i}-\psi_\sigma(f_i)\|_{H^2(\Omega)}^2
\right)<\varepsilon.
\label{eq:hard-tuple-approximation}
\end{equation}
\end{proposition}

\begin{proof}
We divide the proof into four parts.

(1) For the conductivity $\sigma$, set $h=\sigma-A_b$. This function is
continuous and has zero trace. By \eqref{eq:rho-density} in Lemma~\ref{lem:rho-density}, we can choose
continuous functions $v_k$ such that $\rho_\Omega v_k\to h$ uniformly as
$k\to\infty$. By
Lemma~\ref{thm:neural-approximation}, for each $k$ we can choose a finite
network $N_k$ satisfying
\[
\|N_k-v_k\|_{C(\overline\Omega)}<k^{-1}.
\]
Since $\rho_\Omega$ is bounded,
\begin{align*}
&\|(A_b+\rho_\Omega N_k)-\sigma\|_{C(\overline\Omega)}=
\|(\rho_\Omega N_k-\rho_\Omega v_k)+(\rho_\Omega v_k - (\sigma-A_b))\|_{C(\overline\Omega)}\\
&\qquad \leq \|\rho_\Omega\|_{C(\overline\Omega)}
\|N_k-v_k\|_{C(\overline\Omega)}
+\|\rho_\Omega v_k-h\|_{C(\overline\Omega)}\to 0.
\end{align*}
Set
\[
\sigma_k=\clip_{[c_1,c_2]}\!\left(A_b+\rho_\Omega N_k\right).
\]
Since $\clip_{[c_1,c_2]}(\sigma)=\sigma$ and the clipping map has Lipschitz
constant one, we obtain
\[
\|\sigma_k-\sigma\|_{C(\overline\Omega)}=\|\clip_{[c_1,c_2]}(A_b+\rho_\Omega N_k)-\clip_{[c_1,c_2]}(\sigma)\|_{C(\overline\Omega)}
\leq
\|(A_b+\rho_\Omega N_k)-\sigma\|_{C(\overline\Omega)}
\to 0.
\]

(2) To approximate the potential $u_\sigma(f_i)$ for each $i$, we set
$w_i=u_\sigma(f_i)-F_i$. Both $u_\sigma(f_i)$ and $F_i$ have trace $f_i$,
so $w_i\in H^2(\Omega)\cap H_0^1(\Omega)$. By
\eqref{eq:rho-H2-density} in Lemma~\ref{lem:rho-density}, choose $q_{i,k}\in C^2(\overline\Omega)$ such that,
as $k\to\infty$,
\begin{equation*}
\rho_\Omega q_{i,k}\to w_i
\qquad\text{in }H^2(\Omega).
\end{equation*}
For each $k$, Lemma~\ref{thm:neural-approximation} provides a finite network
$U_{i,k}$ satisfying
$\|U_{i,k}-q_{i,k}\|_{C^2(\overline\Omega)}<k^{-1}$. 
Set
\[
u_{i,k}=F_i+\rho_\Omega U_{i,k}.
\]
Since
$\rho_\Omega\in W^{2,\infty}(\Omega)$, it follows that
\begin{align*}
&\|u_{i,k}-u_\sigma(f_i)\|_{H^2(\Omega)}=\|\rho_\Omega U_{i,k}-w_i\|_{H^2(\Omega)}
=\|\rho_\Omega U_{i,k} -\rho_\Omega q_{i,k}+\rho_\Omega q_{i,k}-w_i\|_{H^2(\Omega)}\\
&\qquad\leq
\|\rho_\Omega(U_{i,k}-q_{i,k})\|_{H^2(\Omega)}
+\|\rho_\Omega q_{i,k}-w_i\|_{H^2(\Omega)}\\
&\qquad\leq
C\|\rho_\Omega\|_{W^{2,\infty}(\Omega)}
\|U_{i,k}-q_{i,k}\|_{C^2(\overline\Omega)}
+\|\rho_\Omega q_{i,k}-w_i\|_{H^2(\Omega)}\to0.
\end{align*}

(3) For the stream function $\psi_\sigma(f_i)$, set
$z_i=\psi_\sigma(f_i)-B_i\in H^2(\Omega)$. Then $z_i(x_\star)=0$.
Since $C^2(\overline\Omega)$ is dense in $H^2(\Omega)$ on the given domain,
we can choose $p_{i,k}\in C^2(\overline\Omega)$ such that
$p_{i,k}\to z_i$ in $H^2(\Omega)$ as $k\to\infty$.
By Lemma~\ref{thm:neural-approximation}, choose a finite network $P_{i,k}$ satisfying
$\|P_{i,k}-p_{i,k}\|_{C^2(\overline\Omega)}<k^{-1}$. Set
\[
\psi_{i,k}(x)=B_i+P_{i,k}(x)-P_{i,k}(x_\star).
\]
We first estimate
$P_{i,k}(x_\star)$. Using the embedding
$H^2(\Omega)\hookrightarrow C(\overline{\Omega})$ and $z_i(x_\star)=0$, we obtain
\begin{align*}
&|P_{i,k}(x_\star)|=|P_{i,k}(x_\star)-z_i(x_\star)|\leq |P_{i,k}(x_\star)-p_{i,k}(x_\star)|+|p_{i,k}(x_\star)-z_i(x_\star)|\\
&\qquad \leq  \|P_{i,k}(x)-p_{i,k}(x)\|_{C(\overline\Omega)}+\|p_{i,k}(x)-z_i(x)\|_{C(\overline\Omega)}\leq k^{-1}+C\|p_{i,k}-z_i\|_{H^2(\Omega)}\to0.
\end{align*}
Therefore
\begin{align*}
&\|\psi_{i,k}-\psi_\sigma(f_i)\|_{H^2(\Omega)}=\|P_{i,k}(x)-P_{i,k}(x_\star)-z_i(x)\|_{H^2(\Omega)}\leq \|P_{i,k}-z_i\|_{H^2(\Omega)}+\|P_{i,k}(x_\star)\|_{H^2(\Omega)}\\
&\qquad \leq \|P_{i,k}-p_{i,k}\|_{H^2(\Omega)}
+\|p_{i,k}-z_i\|_{H^2(\Omega)}
+|P_{i,k}(x_\star)|\|1\|_{H^2(\Omega)}\to0.
\end{align*}

(4) Parts (2) and (3) establish the two field convergences in
\eqref{eq:hard-tuple-approximation}. The two scalar approximations can be
placed in the two output channels of one network. Since $M$ is finite,
a single value of $k$ makes all three terms smaller than the prescribed
tolerance.
\end{proof}

\begin{proposition}
\label{prop:neural-zero-loss-approximation}
Let $\sigma\in\mathcal M$, and let $B_i\in H^2(\Omega)$,
$i=1,\ldots,M$, be any fixed stream baselines satisfying
$B_i(x_\star)=0$. For every $\eta>0$, there are finite neural network architectures,
$R>0$, and $\Theta\in\mathcal K_R$ such that
\begin{equation}
\J(\Theta)=\mathcal L(W^\Theta)<\eta.
\label{eq:best-loss-approaches-zero}
\end{equation}
\end{proposition}

\begin{proof}
Apply Proposition~\ref{prop:constrained-field-approximation} with tolerances
tending to zero, and denote the resulting trial fields by
$\sigma_k$, $u_{i,k}$, and $\psi_{i,k}$. Then, as $k\to\infty$,
\[
\sigma_k\to\sigma
\quad\text{in }C(\overline\Omega),
\qquad
u_{i,k}\to u_\sigma(f_i)
\quad\text{in }H^2(\Omega),\qquad
\psi_{i,k}\to\psi_\sigma(f_i)
\quad\text{in }H^2(\Omega),
\]
for each $i=1,\ldots,M$. The clipped conductivity map gives that
$\|\sigma_k\|_{C(\overline\Omega)}\leq c_2$ for every $k$.
By $\nabla^\perp\psi_\sigma(f_i)=\sigma\nabla u_\sigma(f_i)$, we can write
\[
s_i(W^{\Theta_k})=\nabla^\perp\psi_{i,k}
-\sigma_k \nabla u_{i,k}
=\nabla^\perp(\psi_{i,k}-\psi_\sigma(f_i))
-\sigma_k\nabla(u_{i,k}-u_\sigma(f_i))-(\sigma_k-\sigma)\nabla u_\sigma(f_i).
\]
Therefore,
\begin{align*}
\|s_i(W^{\Theta_k})\|_{L^2(\Omega)}
&\leq \|\nabla(\psi_{i,k}-\psi_\sigma(f_i))\|_{L^2(\Omega)}
+\|\sigma_k\|_{C(\overline\Omega)}
 \|\nabla(u_{i,k}-u_\sigma(f_i))\|_{L^2(\Omega)}\\
&\quad+\|\sigma_k-\sigma\|_{C(\overline\Omega)}
 \|\nabla u_\sigma(f_i)\|_{L^2(\Omega)}
\to 0.
\end{align*}
On the boundary, $\Tr\sigma_k=\Tr\sigma=\sigma_b$, so
\[
\Tr s_i(W^{\Theta_k})
=\Tr\!\left(\nabla^\perp(\psi_{i,k}-\psi_\sigma(f_i))\right)
-\sigma_b\,\Tr\!\left(\nabla(u_{i,k}-u_\sigma(f_i))\right).
\]
The embedding $\mathcal H^{1/2}(\partial\Omega)\hookrightarrow L^2(\partial\Omega)$
and the trace theorem imply
\begin{align*}
&\|\Tr s_i(W^{\Theta_k})\|_{L^2(\partial\Omega)}\leq C\|\Tr s_i(W^{\Theta_k})\|_{\mathcal H^{1/2}(\partial\Omega)}\\
\leq &C\|\ \nabla^\perp(\psi_{i,k}-\psi_\sigma(f_i))\|_{H^1(\Omega)}+C\|\nabla(u_{i,k}-u_\sigma(f_i))\|_{H^1(\Omega)}\to 0.
\end{align*}
Since $\sigma\in\mathcal M$, the exact current satisfies
\[
g_i=\Lambda_\sigma f_i=\partial_\tau\psi_\sigma(f_i).
\]
Thus
\[
d_i^0(W^{\Theta_k})
=\partial_\tau\psi_{i,k}-g_i
=\partial_\tau(\psi_{i,k}-\psi_\sigma(f_i)),
\]
and therefore
\[
\|d_i^0(W^{\Theta_k})\|_{L^2(\partial\Omega)}\leq \|d_i^0(W^{\Theta_k})\|_{\mathcal H^{1/2}(\partial\Omega)}
\leq C\|\psi_{i,k}-\psi_\sigma(f_i)\|_{H^2(\Omega)}\to 0.
\]
It follows that $\mathcal L(W^{\Theta_k})\to 0$. Choose $k$ so that the loss is less
than $\eta$. The corresponding networks have a finite parameter vector
$\Theta_k\in\R^{N_{\mathrm{par},k}}$, where $N_{\mathrm{par},k}$ is the total
number of weights and biases in the networks chosen for the $k$th
approximation. Setting
$R_k=1+\|\Theta_k\|_\infty$ places this vector in the compact box
$[-R_k,R_k]^{N_{\mathrm{par},k}}$, which proves
\eqref{eq:best-loss-approaches-zero}.
\end{proof}

\subsection{The sampled problem and convergence}
\label{sec:discrete-inexact-loss}
We approximate the three integrals defining $\mathcal L^\delta$ in
\eqref{eq:stream-residual-loss} by weighted sums. The residuals are evaluated
at collocation points, with interior points used for $s_i$ and boundary
points for $\Tr s_i$ and $d_i^\delta$. The two boundary point sets may
coincide. We consider both Monte Carlo sampling and deterministic quadrature.
We first establish
sampling estimates and the existence of minimizers for the sampled
problems, then derive error estimates and convergence results for the
associated conductivity reconstructions.

Define
\[
\mathcal I_{\Omega,N_\Omega}[a]:=\sum_{j=1}^{N_\Omega}w_{\Omega,j}\, a(x^{(j)}),\qquad
\mathcal I_{b,N_b}[b]:=\sum_{\ell=1}^{N_b}w_{b,\ell}\, b(y^{(\ell)}),\qquad
\mathcal I_{q,N_q}[c]:=\sum_{r=1}^{N_q}w_{q,r}\, c(z^{(r)}),
\]
where $N_\Omega$, $N_b$, 
and $N_q$ are the respective point counts, $w_{\Omega,j}$, $w_{b,\ell}$, and $w_{q,r}$ are nonnegative weights,
and $x^{(j)}\in\Omega$, $y^{(\ell)}\in\partial\Omega$, and
$z^{(r)}\in\partial\Omega$ are the corresponding sampling points. 
The points and weights are specified for each rule below. With
$\mathbf N=(N_\Omega,N_b,N_q)$, the sampled loss is
\begin{equation}
\J_{\mathbf N}^\delta(\Theta)
=\frac1M\sum_{i=1}^M\Big(
\mathcal I_{\Omega,N_\Omega}[|s_i(W^\Theta)|^2]
+\lambda_b\, \mathcal I_{b,N_b}[|\Tr s_i(W^\Theta)|^2]
+\lambda_q\, \mathcal I_{q,N_q}[|d_i^\delta(W^\Theta)|^2]
\Big).
\label{eq:theoretical-sampled-core}
\end{equation}
On the compact parameter set $\mathcal K_R$, define the uniform sampling error
\begin{equation*}
\varepsilon_{\mathrm{sam},\mathbf N}^\delta
:=\sup_{\Theta\in\mathcal K_R}
|\J_{\mathbf N}^\delta(\Theta)-\J^\delta(\Theta)|.
\end{equation*}

We first consider Monte Carlo sampling. Let $(x^{(j)})_{j\geq1}$ consist of
independent uniform samples from $\Omega$, and let each of
$(y^{(\ell)})_{\ell\geq1}$ and $(z^{(r)})_{r\geq1}$ consist of independent uniform
samples in arclength from $\partial\Omega$. Set
\[
w_{\Omega,j}=\frac{|\Omega|}{N_\Omega},\quad 
w_{b,\ell}=\frac{|\partial\Omega|}{N_b}, \quad \text{and}\quad
w_{q,r}=\frac{|\partial\Omega|}{N_q}.
\]

\begin{proposition}
\label{prop:uniform-sampling}
For the Monte Carlo sampling defined above, the following statements hold.
\begin{enumerate}[label=(\arabic*)]
\item We have
\begin{equation}
\varepsilon_{\mathrm{sam},\mathbf N}^\delta\to 0
\qquad\text{almost surely as }
\min\{N_\Omega,N_b,N_q\}\to\infty.
\label{eq:proved-uniform-sampling}
\end{equation}
\item For every fixed $\Theta\in\mathcal K_R$, the Monte Carlo loss is unbiased and
\begin{equation}
\mathbb E\bigl[\J_{\mathbf N}^\delta(\Theta)\bigr]
=\J^\delta(\Theta),\qquad
\operatorname{Var}\bigl(\J_{\mathbf N}^\delta(\Theta)\bigr)
\leq C_{R,\delta}
\left(N_\Omega^{-1}+N_b^{-1}+N_q^{-1}\right).
\label{eq:mc-fixed-moments}
\end{equation}
\item For every finite $\mathbf N$, $\J_{\mathbf N}^\delta$ is continuous on
$\mathcal K_R$ and attains a minimum.
\end{enumerate}
\end{proposition}

\begin{proof} 
The proof is given in Appendix~\ref{app:uniform-sampling}.
\end{proof}

We next consider deterministic quadrature. Let $\mathcal T_\Omega$ be a fitted
partition of $\Omega$ into $N_\Omega$ cells. Let $\mathcal E_b$ and $\mathcal E_q$ be arclength partitions of
$\partial\Omega$ into $N_b$ and $N_q$ intervals, respectively. In the
curvilinear polygonal case, these partitions are formed separately on each
edge. We choose one quadrature point in each cell and interval, with weights
given by the corresponding areas $|K|$ and arclengths $|E|$. In the
curvilinear polygonal case, boundary quadrature points are not placed at
the vertices. Define
\[
h_\Omega=\max_{K\in\mathcal T_\Omega}\operatorname{diam}(K),
\qquad
h_b=\max_{E\in\mathcal E_b}|E|,
\qquad
h_q=\max_{E\in\mathcal E_q}|E|.
\]
We use partitions satisfying
\[
h_\Omega\leq C_{\rm pt}N_\Omega^{-1/2},
\qquad
h_b\leq C_{\rm pt}N_b^{-1},
\qquad
h_q\leq C_{\rm pt}N_q^{-1},
\]
where $C_{\rm pt}$ is independent of the point counts.

\begin{proposition}
\label{prop:det-midpoint}
Suppose that $f_i\in C^{1,1}(\partial\Omega)$ and
$g_i^\delta\in C^{0,1}(\partial\Omega)$, $i=1,\ldots,M$.
In the curvilinear polygonal case, these boundary regularity conditions are
understood edgewise. The fixed potential liftings $F_i$ and stream
baselines $B_i$ are chosen in $W^{2,\infty}(\Omega)$.
For the deterministic quadrature defined above, the following statements hold.
\begin{enumerate}[label=(\arabic*)]
\item We have
\begin{equation}
\varepsilon_{\mathrm{sam},\mathbf N}^\delta
\leq C(h_\Omega+h_b+h_q)
\leq C
\bigl(N_\Omega^{-1/2}+N_b^{-1}+N_q^{-1}\bigr),
\label{eq:deterministic-sampling-consistency}
\end{equation}
where the constant $C$ is independent of the point counts.
\item For every finite $\mathbf N$, $\J_{\mathbf N}^\delta$ is continuous on
$\mathcal K_R$ and attains a minimum.
\end{enumerate}
\end{proposition}

\begin{proof} 
The proof is given in Appendix~\ref{app:uniform-sampling}.
\end{proof}

\begin{remark}
\label{rem:deterministic-lifting-regularity}
The stronger boundary regularity in Proposition~\ref{prop:det-midpoint} is
required only for deterministic quadrature. It allows the potential liftings
$F_i$ and stream baselines $B_i$ to be chosen in $W^{2,\infty}(\Omega)$,
so that the loss integrands are Lipschitz in the spatial variable and the first order 
quadrature estimate applies. Throughout the subsequent analysis, we assume the hypotheses of
Proposition~\ref{prop:det-midpoint} whenever deterministic quadrature is used.
\end{remark}

As the measurement noise tends to zero, we next choose the neural classes
and sample counts so that the neural approximation error and sampling error
also tend to zero. Let
$\delta_n\to0^+$, $g_i^{\delta_n}$ satisfy \eqref{eq:noise-model}, and
$B_{i,n}\in H^2(\Omega)$ be the corresponding fixed stream baselines.
Fix $\bar\sigma\in\mathcal M$. Then
\[
g_i=\Lambda_{\bar\sigma}f_i
=\partial_\tau\psi_{\bar\sigma}(f_i),
\qquad i=1,\ldots,M.
\]
Let $\eta_n\to 0^+$. For each $n$, applying
Proposition~\ref{prop:neural-zero-loss-approximation} to $\bar\sigma$ with
baselines $B_{i,n}$ and tolerance $\eta_n$, we obtain a finite neural network
architecture, $R_n>0$, and a parameter vector $\Theta_n^0$ satisfying
\[
\Theta_n^0\in\mathcal K_n:=[-R_n,R_n]^{N_{\mathrm{par},n}},
\qquad
\J(\Theta_n^0)<\eta_n.
\]
For $\Theta_n\in\mathcal K_n$, let $W_n^{\Theta_n}$ denote the corresponding
trial tuple, and write $\J_n(\Theta_n)=\mathcal L(W_n^{\Theta_n})$ for the exact data
loss on this neural class. Define the neural approximation error by
\begin{equation}
e_{\mathrm{NN},n}
:=\min_{\Theta_n\in\mathcal K_n}\J_n(\Theta_n).
\label{eq:NN-reference-approximation-error}
\end{equation}
By Proposition~\ref{prop:fixed-architecture-minimum}, $\J_n$ is continuous on the compact set ${\mathcal K_n}$, 
so the minimum is attained. Since $\Theta_n^0\in\mathcal K_n$ and $\J_n(\Theta_n^0)<\eta_n$, we have
\[
0\leq e_{\mathrm{NN},n}\leq\J_n(\Theta_n^0)<\eta_n.
\]
Therefore,
\[
e_{\mathrm{NN},n}\to0
\qquad\text{as }n\to\infty.
\]

For each $n$, we use the parameter set $\mathcal K_n$ in place of
$\mathcal K_R$ defined in \eqref{eq:parameter-box}. For
$\Theta_n\in\mathcal K_n$, define
\begin{equation*}
\J_n^{\delta_n}(\Theta_n)=\mathcal L^{\delta_n}(W_n^{\Theta_n}),
\end{equation*}
where $\mathcal L^{\delta_n}$ is given by \eqref{eq:stream-residual-loss} with $\delta=\delta_n$ and 
data $g_i^{\delta_n}$. For any
triple of point counts $\mathbf N$, let $\J_{n,\mathbf N}^{\delta_n}$ denote
the sampled loss on $\mathcal K_n$ obtained from
\eqref{eq:theoretical-sampled-core} with trial tuple $W_n^{\Theta_n}$. Define the corresponding 
uniform sampling error by
\begin{equation*}
\varepsilon_{\mathrm{sam},n,\mathbf N}^{\delta_n}
=\sup_{\Theta_n\in\mathcal K_n}
\left|\J_{n,\mathbf N}^{\delta_n}(\Theta_n)
-\J_n^{\delta_n}(\Theta_n)\right|.
\end{equation*}

Propositions~\ref{prop:uniform-sampling} and~\ref{prop:det-midpoint} give
sampling convergence for each fixed neural class and noise level. The neural
class and noise level now vary with $n$. However, for the convergence
analysis below, we need the sampling errors to tend to zero as $n\to\infty$.
The following proposition provides the required choice of point counts for
each $n$.

\begin{proposition}
\label{prop:sampling-selection}
For either rule, there exist triples
$\mathbf N_n=(N_{\Omega,n},N_{b,n},N_{q,n})$ such that, as $n\to\infty$,
\[
\min\{N_{\Omega,n},N_{b,n},N_{q,n}\}\to\infty
\quad\text{and}\quad
\varepsilon_{\mathrm{sam},n,\mathbf N_n}^{\delta_n}\to0.
\]
The second limit holds almost surely for Monte Carlo sampling and deterministically for
deterministic quadrature.
\end{proposition}

\begin{proof}
The proof is given in Appendix~\ref{app:uniform-sampling}.
\end{proof}

We now fix a sequence of point counts given by
Proposition~\ref{prop:sampling-selection}. Then
\[
\varepsilon_{\mathrm{sam},n,\mathbf N_n}^{\delta_n}\to0
\]
almost surely for Monte Carlo sampling and deterministically for quadrature.
For these point counts, consider the sampled minimization problem
\[
\min_{\Theta_n\in\mathcal K_n}
\J_{n,\mathbf N_n}^{\delta_n}(\Theta_n).
\]
By Proposition~\ref{prop:uniform-sampling} and Proposition~\ref{prop:det-midpoint},
this problem admits a minimizer for either sampling rule. 
Let $\widehat\Theta_n$ denote any such minimizer.
We denote by $\widehat\vartheta_n$ the conductivity network parameters in
$\widehat\Theta_n$ and set $\widehat\sigma_n=\sigma^{\widehat\vartheta_n}$.

\begin{theorem}
\label{thm:neural-solution-set-convergence}
There is a constant $C>0$ independent of $n$ such that
\begin{equation}
\mathcal R(\widehat\sigma_n)
\leq C\left(
e_{\mathrm{NN},n}+\delta_n^2
+\varepsilon_{\mathrm{sam},n,\mathbf N_n}^{\delta_n}
\right).
\label{eq:sampled-current-residual}
\end{equation}
In particular, $\mathcal R(\widehat\sigma_n)\to0$.
If, in addition, Assumption~\ref{ass:solution-set-stability} holds, then
\begin{equation}
\operatorname{dist}_{C(\overline\Omega)}
(\widehat\sigma_n,\mathcal M)
\leq
\omega\!\left(
C\left(
 e_{\mathrm{NN},n}
 +\delta_n^2
 +\varepsilon_{\mathrm{sam},n,\mathbf N_n}^{\delta_n}
\right)\right),
\label{eq:neural-error-decomposition}
\end{equation}
and consequently
$\operatorname{dist}_{C(\overline\Omega)}
(\widehat\sigma_n,\mathcal M)\to 0$.
These convergences are almost sure for Monte Carlo sampling and deterministic
for deterministic quadrature.
\end{theorem}

\begin{proof}
We first compare the exact and noisy losses. Let $\Theta_n^\ast\in\mathcal K_n$
attain the minimum in \eqref{eq:NN-reference-approximation-error}. Using
\[
\partial_\tau\psi_{i,n}^{\Theta_n^\ast}-g_i^{\delta_n}
=\bigl(\partial_\tau\psi_{i,n}^{\Theta_n^\ast}-g_i\bigr)
+\bigl(g_i-g_i^{\delta_n}\bigr)
\]
and $|a+b|^2\leq2|a|^2+2|b|^2$, we obtain
\begin{align*}
\J_n^{\delta_n}(\Theta_n^\ast)
&=\frac1M\sum_{i=1}^M\Bigl(
\|s_i(W_n^{\Theta_n^\ast})\|_{L^2(\Omega)}^2
+\lambda_b\|\Tr s_i(W_n^{\Theta_n^\ast})\|_{L^2(\partial\Omega)}^2
+\lambda_q\|\partial_\tau\psi_{i,n}^{\Theta_n^\ast}
-g_i^{\delta_n}\|_{L^2(\partial\Omega)}^2
\Bigr)\\
&\leq\frac1M\sum_{i=1}^M\Bigl(
\|s_i(W_n^{\Theta_n^\ast})\|_{L^2(\Omega)}^2
+\lambda_b\|\Tr s_i(W_n^{\Theta_n^\ast})\|_{L^2(\partial\Omega)}^2\\
&\qquad\qquad
+2\lambda_q\|\partial_\tau\psi_{i,n}^{\Theta_n^\ast}
-g_i\|_{L^2(\partial\Omega)}^2
+2\lambda_q\|g_i-g_i^{\delta_n}\|_{L^2(\partial\Omega)}^2
\Bigr)\\
&\leq2 \J_n(\Theta_n^\ast)
+\frac{2\lambda_q}{M}\sum_{i=1}^M
\|g_i-g_i^{\delta_n}\|_{L^2(\partial\Omega)}^2
\leq2e_{\mathrm{NN},n}+2\lambda_q\delta_n^2.
\end{align*}
The last inequality follows from the minimizing property of $\Theta_n^\ast$
in \eqref{eq:NN-reference-approximation-error}
and the noise bound \eqref{eq:noise-model}.
By the definition of
$\varepsilon_{\mathrm{sam},n,\mathbf N_n}^{\delta_n}$ and the minimality of
$\widehat\Theta_n$, we know
\begin{align*}
\J_n^{\delta_n}(\widehat\Theta_n)
&\leq\J_{n,\mathbf N_n}^{\delta_n}(\widehat\Theta_n)
+\varepsilon_{\mathrm{sam},n,\mathbf N_n}^{\delta_n}
\leq\J_{n,\mathbf N_n}^{\delta_n}(\Theta_n^\ast)
+\varepsilon_{\mathrm{sam},n,\mathbf N_n}^{\delta_n}\\
&\leq\J_n^{\delta_n}(\Theta_n^\ast)
+2\varepsilon_{\mathrm{sam},n,\mathbf N_n}^{\delta_n}
\leq2e_{\mathrm{NN},n}+2\lambda_q\delta_n^2
+2\varepsilon_{\mathrm{sam},n,\mathbf N_n}^{\delta_n}.
\end{align*}
Finally, applying Theorem~\ref{lem:stream-current} yields
\eqref{eq:sampled-current-residual}. Since $e_{\mathrm{NN},n}\to0$,
$\delta_n\to0$, and $\varepsilon_{\mathrm{sam},n,\mathbf N_n}^{\delta_n}\to0$,
estimate \eqref{eq:sampled-current-residual} implies $\mathcal R(\widehat\sigma_n)\to0$.
Corollary~\ref{cor:stream-conductivity} then gives
\eqref{eq:neural-error-decomposition} and $\operatorname{dist}_{C(\overline\Omega)}
(\widehat\sigma_n,\mathcal M)\to 0$. The almost sure and deterministic convergence statements 
follow from the corresponding sampling error convergence in Proposition~\ref{prop:sampling-selection}.
\end{proof}

\section{Numerical implementation and experiments}
\label{sec:stream-numerics}

In this section, we present four numerical examples on $\Omega=(-1,1)^2$
to assess the reconstruction accuracy of the proposed method with exact
and noisy data. All examples use the same four boundary voltages, chosen
as the linearly independent sine and cosine functions
\begin{equation*}
\begin{aligned}
f_1(x_1,x_2)&=\sin(x_1+x_2),& f_2(x_1,x_2)&=\cos(x_1+x_2),\\
f_3(x_1,x_2)&=\sin\bigl(2(x_1+x_2)\bigr),&
f_4(x_1,x_2)&=\cos\bigl(2(x_1+x_2)\bigr).
\end{aligned}
\end{equation*}
Additional measurement pairs may improve reconstruction accuracy, but
they also increase the computational cost.
All examples also use the same network architectures, collocation points,
and procedures for constructing the trial fields and training the networks.
We use $50{,}000$ training epochs for Examples 1--3 and $150{,}000$ for
Example 4.

\subsection{Boundary data and sampling}
\label{subsec:numerical-setting}

For each true conductivity $\sigma_{\mathrm{true}}$, we generate synthetic
boundary data using MATLAB PDE Toolbox. We solve the forward problem for
$u_i$ with boundary voltage $f_i$ and compute the corresponding boundary
current $g_i=\sigma_{\mathrm{true}}\partial_n u_i$ on $\partial\Omega$.
The toolbox uses quadratic triangular finite elements with the maximum 
mesh size set to $0.002$. The boundary currents 
are evaluated at $200$ equally
spaced midpoints on each side of the square using the interpolated current
density. These $800$ points form the boundary collocation set $Q$,
with arclength spacing $\Delta s=0.01$.

The boundary current data used in the experiments are given by
\begin{equation*}
g_{i,\epsilon}(z)=g_i(z)+\epsilon\,|g_i(z)|\,\xi_i(z),
\qquad z\in Q,\quad i=1,\ldots,4,
\end{equation*}
where the $\xi_i(z)$ are independent standard normal random variables. We use
$\epsilon=0$ for exact data and $\epsilon=0.01$ for $1\%$ pointwise relative
noise. Here $\epsilon$ controls the pointwise relative noise level in the
numerical data, whereas $\delta$ in \eqref{eq:noise-model} bounds the measurement
error in the boundary $L^2$ norm.

The interior collocation set $X$ consists of the centers of the
$100\times100$ cells in a uniform partition of $\Omega$, with spacing
$h=0.02$. The interior constitutive residual is evaluated on $X$, and
the boundary constitutive and current residuals are evaluated on $Q$.

\subsection{Neural networks and training}
\label{subsec:numerical-networks}
\label{subsec:numerical-optimization}

We use the networks introduced in Section~\ref{sec:neural-architectures}.
Each of the four field networks has widths $2\to64\to64\to64\to2$,
and the shared conductivity network has widths
$2\to64\to64\to64\to64\to1$. We use $\tanh$ as the activation function.

For the trial functions in \eqref{eq:conductivity-hard-map},
\eqref{eq:field-hard-representations}, and
\eqref{eq:stream-network-representation}, we choose
\begin{equation*}
\rho(x_1,x_2)=(1-x_1^2)(1-x_2^2),\qquad
c_1=0.1,\quad c_2=10,\qquad
x_\star=(-1,-1).
\end{equation*}
We use Coons interpolation
\cite{GordonHall1973} to extend the known boundary conductivity
$\sigma_b$ into the domain $\Omega=(-1,1)^2$.
For calculations on the square, write $\xi=(x_1+1)/2$ and
$\zeta=(x_2+1)/2$. We first define the bilinear corner interpolant of $\sigma_b$ by
\begin{equation}
\begin{aligned}
p_{\sigma_b}(x_1,x_2)={}&(1-\xi)(1-\zeta)\sigma_b(-1,-1)
          +\xi(1-\zeta)\sigma_b(1,-1)\\
          &+\xi\zeta\sigma_b(1,1)+(1-\xi)\zeta\sigma_b(-1,1).
\end{aligned}
\label{eq:numerical-corner-interpolant}
\end{equation}
Let $\sigma_L,\sigma_R,\sigma_B,\sigma_T$ denote the prescribed boundary
conductivity $\sigma_b$ on the left, right, bottom, and top sides of the
square, respectively. Then the Coons extension for the conductivity lifting is given by
\begin{equation*}
A_b(x_1,x_2)=(1-\xi)\sigma_L(x_2)+\xi\sigma_R(x_2)+(1-\zeta)\sigma_B(x_1)+\zeta\sigma_T(x_1)-p_{\sigma_b}(x_1,x_2).
\end{equation*}
This extension agrees with the prescribed conductivity on each side.
We use harmonic extensions for the potential and stream liftings.
Details of these constructions are given in
Appendix~\ref{app:numerical-details}.

Experiments with larger current magnitudes can dominate the loss.
To normalize the contributions of the four experiments, we set
\begin{equation*}
\kappa_i^2=\frac{1}{|Q|}\sum_{z\in Q}|g_{i,\epsilon}(z)|^2
                +c_{\mathrm{cur}},
\qquad c_{\mathrm{cur}}=10^{-8}.
\end{equation*}
These factors remain fixed during training. The normalized sampled
counterparts of $\mathcal L_\Omega$, $\mathcal L_b$, and $\mathcal L_q^\delta$
in Section~\ref{sec:continuous-neural-loss} are
\begin{equation*}
\begin{aligned}
\widetilde{\mathcal L}_\Omega(\Theta)
 &=\frac{1}{4|X|}\sum_{i=1}^4\sum_{x\in X}
       \frac{|s_i(W^\Theta)(x)|^2}{\kappa_i^2},\\
\widetilde{\mathcal L}_b(\Theta)
 &=\frac{1}{4|Q|}\sum_{i=1}^4\sum_{y\in Q}
       \frac{|s_i(W^\Theta)(y)|^2}{\kappa_i^2},\\
\widetilde{\mathcal L}_q(\Theta)
 &=\frac{1}{4|Q|}\sum_{i=1}^4\sum_{z\in Q}
       \frac{|\partial_\tau\psi_i^{\theta_i}(z)-g_{i,\epsilon}(z)|^2}
            {\kappa_i^2}.
\end{aligned}
\end{equation*}
To accelerate convergence during training, we also use a dynamic weighting
strategy \cite{WangSankaranWangPerdikaris2023,WangTengPerdikaris2021}. At epoch $t$, the training loss is
\begin{equation*}
\widetilde{\mathcal L}(\Theta;t)
=\widetilde{\mathcal L}_\Omega(\Theta)
 +\alpha_b(t)\widetilde{\mathcal L}_b(\Theta)
 +\alpha_q(t)\widetilde{\mathcal L}_q(\Theta).
\end{equation*}
For the dynamic updates, we use gradients with respect to the field network
parameters $\theta=(\theta_1,\theta_2,\theta_3,\theta_4)$.
We do not include the conductivity network parameters in these gradient
norms, since only the interior loss depends on them and the two boundary
loss terms are independent of them.
Every $2000$ epochs, we update the coefficients $\alpha_b$ and $\alpha_q$ using the moving average formulas
\begin{equation*}
\alpha_b^{\mathrm{new}}
=0.9\alpha_b^{\mathrm{old}}
 +0.1\frac{\|\nabla_\theta\widetilde{\mathcal L}_\Omega\|_2}{\|\nabla_\theta\widetilde{\mathcal L}_b\|_2+c_{\mathrm{bal}}}
\quad\text{and}\quad
\alpha_q^{\mathrm{new}}
=0.9\alpha_q^{\mathrm{old}}
 +0.1\frac{\|\nabla_\theta\widetilde{\mathcal L}_\Omega\|_2}{\|\nabla_\theta\widetilde{\mathcal L}_q\|_2+c_{\mathrm{bal}}},
\end{equation*}
where $\|\cdot\|_2$ denotes the Euclidean norm and $c_{\mathrm{bal}}=10^{-12}$. The coefficients remain fixed between
updates and are treated as constants when differentiating the total loss.

All networks are trained jointly with Adam \cite{KingmaBa2015}, using all
collocation points and all four voltage experiments at each step. 
We report the results obtained at the final training epoch. The training settings are summarized in
Table~\ref{tab:training-settings}.

\begin{table}[!htbp]
\centering\small
\renewcommand{\arraystretch}{1.2}
\caption{Training settings used in the numerical experiments.}
\label{tab:training-settings}
\begin{tabularx}{\textwidth}{@{}>{\raggedright\arraybackslash}p{0.43\textwidth}Y@{}}
\toprule
Quantity & Setting\\
\midrule
Initial learning rates & $5\times10^{-4}$ for field networks; $10^{-3}$ for the conductivity network\\
Learning rate decay & Factor $0.8$ every $5000$ epochs\\
Initial $(\alpha_b,\alpha_q)$ & $(0.1,1)$\\
Coefficient update interval & $2000$ epochs\\
Moving average factor for coefficient updates & $0.9$\\
Training epochs & $50{,}000$ for Examples 1--3; $150{,}000$ for Example 4\\
\bottomrule
\end{tabularx}
\end{table}

\subsection{Reconstruction errors}
\label{subsec:numerical-metrics}

Let $\sigma=\sigma_{\mathrm{true}}$ and
$\widehat\sigma=\sigma^{\widehat\vartheta}$, where
$\widehat\vartheta$ denotes the final conductivity network parameters. We define
the four reconstruction errors by

\begin{equation*}
\begin{aligned}
E_{\sigma,1}
&=\frac{\int_\Omega|\widehat\sigma-\sigma|\,dx}
        {\int_\Omega|\sigma|\,dx},
&\qquad
E_{\sigma,2}
&=\left(\frac{\int_\Omega|\widehat\sigma-\sigma|^2\,dx}
              {\int_\Omega|\sigma|^2\,dx}\right)^{1/2},\\
\mathrm{RE}_{\mathrm{mean}}
&=\frac1{|\Omega|}\int_\Omega
       \frac{|\widehat\sigma-\sigma|}{|\sigma|}\,dx,
&
\mathrm{MSE}
&=\frac1{|\Omega|}\int_\Omega|\widehat\sigma-\sigma|^2\,dx.
\end{aligned}
\end{equation*}
Here, $E_{\sigma,1}$ and $E_{\sigma,2}$ are the relative $L^1$ and $L^2$
errors. The quantity $\mathrm{RE}_{\mathrm{mean}}$ averages the pointwise
relative error over $\Omega$, and the mean squared error (MSE) averages
the squared error. All four errors are computed on the $100\times100$
interior grid $X$. On the same grid, we also compute the
minimum and maximum reconstructed conductivities
\[
\widehat\sigma_{\min}=\min_{x\in X}\widehat\sigma(x)
\quad\text{and}\quad
\widehat\sigma_{\max}=\max_{x\in X}\widehat\sigma(x).
\]

\newcommand{\numreconfigure}[3][!htbp]{%
\begin{figure}[#1]\centering
\IfFileExists{figures/example#2_all.pdf}{%
\includegraphics[width=\linewidth]{figures/example#2_all.pdf}}{%
\IfFileExists{figures/example#2_all.png}{%
\includegraphics[width=\linewidth]{figures/example#2_all.png}}{%
\fbox{\parbox[c][1.4in][c]{0.95\linewidth}{\centering
Combined figure pending\\[4pt]
\texttt{example#2\_all.pdf}}}}}
\caption{#3. Left: True conductivity; Middle: Reconstruction from exact
data; Right: Reconstruction from $1\%$ noisy data.}
\label{fig:numerical-profile-#2}\end{figure}}

\subsection{Example 1. One blob conductivity}
\label{subsec:num-example-one}

We first consider the reconstruction of a single smooth blob on
$\Omega=(-1,1)^2$, as in Example 5.2 of \cite{LiWang2023}.
The true conductivity is given by
\begin{equation}
\sigma_{\mathrm{true}}^{(1)}(x_1,x_2)
 =1+e^{-8[x_1^2+(x_2-0.55)^2]}.
\label{eq:num-profile-one}
\end{equation}
The blob is centered at $(0,0.55)$, with background conductivity $1$
and height $2$. We consider exact data and $1\%$ noisy data.
Figure~\ref{fig:numerical-profile-1} presents the true conductivity
and the two computed conductivities. The reconstruction errors and
computed heights are listed in Table~\ref{tab:num-example-one},
together with the results of the minimal dissipation local discontinuous
Galerkin method (MD-LDG) in \cite{LiWang2023}.

\numreconfigure[!htbp]{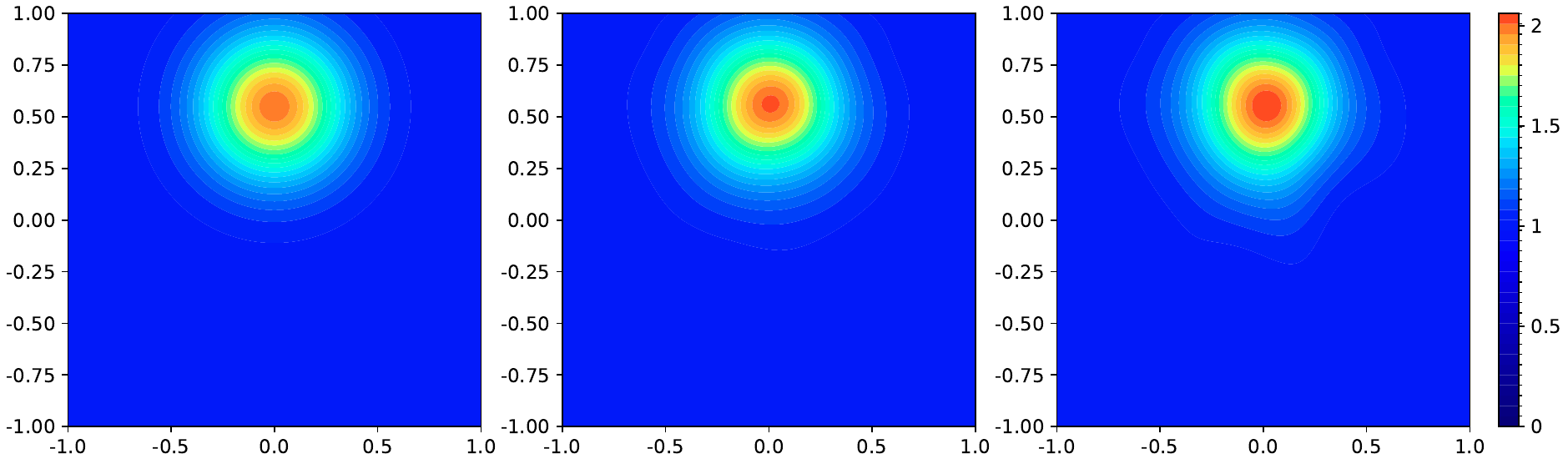}{One blob conductivity \eqref{eq:num-profile-one}}

\begin{table}[!htbp]
\centering\small
\setlength{\tabcolsep}{3.2pt}
\renewcommand{\arraystretch}{1.15}
\settowidth{\reconstructioncaptionwidth}{\normalfont\normalsize
Reconstruction errors and predicted conductivity ranges for exact and $1\%$ noisy data.}
\captionsetup{format=plain,font=normalsize,justification=raggedright,
  singlelinecheck=false,width=\reconstructioncaptionwidth}
\caption[Example 1. Reconstruction errors and predicted conductivity ranges for exact and $1\%$ noisy data.]{Example 1.\\
Reconstruction errors and predicted conductivity ranges for exact and $1\%$ noisy data.}
\label{tab:num-example-one}
\label{tab:num-example-one-comparison}
\begin{tabular}{@{}llcccccc@{}}
\toprule
Data & Method & $E_{\sigma,1}$ & $E_{\sigma,2}$ & $\mathrm{RE}_{\mathrm{mean}}$ & MSE & $\hat\sigma_{\min}$ & $\hat\sigma_{\max}$\\
\midrule
Exact & Our method & 0.003500 & 0.006539 & 0.003050 & 0.00005294 & 0.99439 & 2.02584\\
      & MD-LDG~\cite{LiWang2023} & 0.017813 & 0.034917 & 0.015141 & 0.00150952 & 0.97438 & 1.81642\\
\midrule
$1\%$ & Our method & 0.006564 & 0.012032 & 0.005819 & 0.00017925 & 0.98801 & 2.06058\\
      & MD-LDG~\cite{LiWang2023} & 0.027572 & 0.051694 & 0.023584 & 0.00330859 & 0.97130 & 1.72355\\
\bottomrule
\end{tabular}
\end{table}

Figure~\ref{fig:numerical-profile-1} demonstrates that the proposed method
accurately reconstructs the location, shape, and height of the blob from
both exact and noisy boundary data. The quantitative results in
Table~\ref{tab:num-example-one} confirm the high reconstruction accuracy,
with relative $L^2$ errors of only $0.006539$ and $0.012032$, respectively.
The reconstructed heights are also close to the true value of $2$.
The comparison with the MD-LDG method~\cite{LiWang2023} shows a substantial
improvement in reconstruction accuracy. All four error measures are smaller
for both data conditions, and the relative $L^2$ errors are reduced by
approximately $81\%$ and $77\%$ for exact and noisy data, respectively. The proposed
method also recovers the blob height more accurately.

\subsection{Example 2. Two blobs conductivity}
\label{subsec:num-example-two}

In the second example, we consider two smooth blobs on the same
domain, as in Example 5.3 of \cite{LiWang2023}. The true
conductivity is given by
\begin{equation}
\sigma_{\mathrm{true}}^{(2)}(x_1,x_2)
=1+e^{-20[(x_1+0.7)^2+x_2^2]}+e^{-20[x_1^2+(x_2-0.7)^2]}.
\label{eq:num-profile-two}
\end{equation}

\numreconfigure{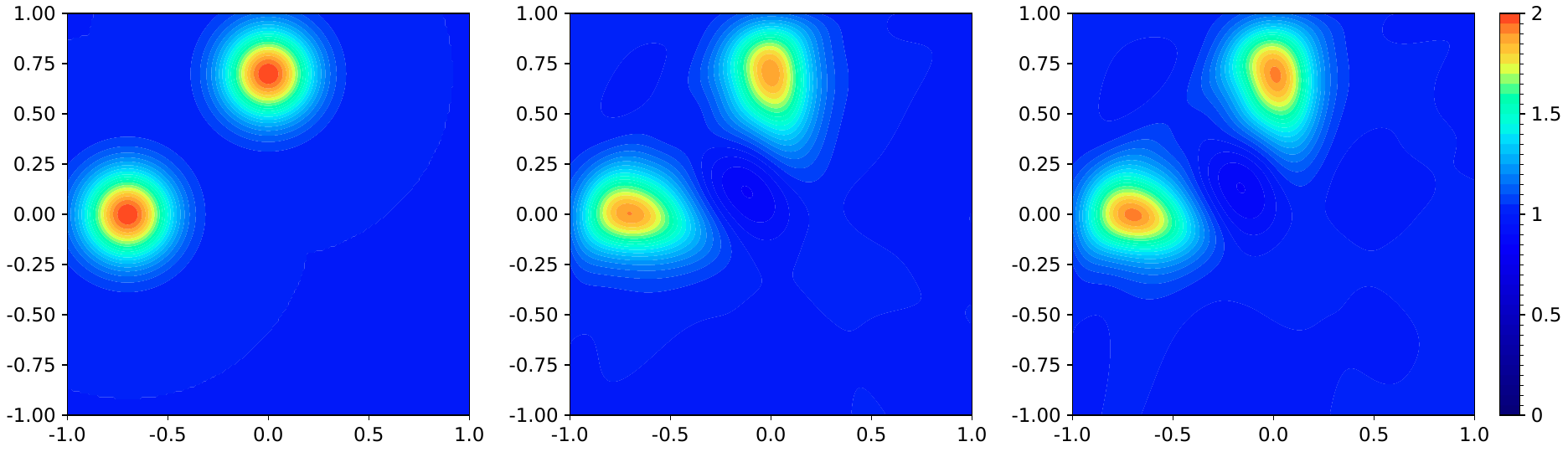}{Two blobs conductivity \eqref{eq:num-profile-two}}

The two blobs are centered at $(-0.7,0)$ and $(0,0.7)$, with
background conductivity $1$ and height approximately $2$.
Figure~\ref{fig:numerical-profile-2} presents the true conductivity
and the computed conductivities for exact and $1\%$ noisy data.
It demonstrates that the proposed method
successfully resolves the two blobs and accurately identifies their
locations from both exact and noisy boundary data. The blobs remain clearly
separated in both reconstructions. Table~\ref{tab:num-example-two} lists the reconstruction errors and
computed conductivity ranges for both noise levels. It shows
relative $L^2$ errors of $0.037354$ and $0.036223$. 
The reconstructed maxima, $1.90126$ and $1.92163$, also provide good
approximations of the true maximum of approximately $2$.

\begin{table}[!htbp]
\centering\small
\setlength{\tabcolsep}{4pt}
\renewcommand{\arraystretch}{1.15}
\settowidth{\reconstructioncaptionwidth}{\normalfont\normalsize
Reconstruction errors and predicted conductivity ranges for exact and $1\%$ noisy data.}
\captionsetup{format=plain,font=normalsize,justification=raggedright,
  singlelinecheck=false,width=\reconstructioncaptionwidth}
\caption[Example 2. Reconstruction errors and predicted conductivity ranges for exact and $1\%$ noisy data.]{Example 2.\\
Reconstruction errors and predicted conductivity ranges for exact and $1\%$ noisy data.}
\label{tab:num-example-two}
\begin{tabular}{@{}lcccccc@{}}
\toprule
Data & $E_{\sigma,1}$ & $E_{\sigma,2}$ & $\mathrm{RE}_{\mathrm{mean}}$ & MSE & $\hat\sigma_{\min}$ & $\hat\sigma_{\max}$\\
\midrule
Exact & 0.017469 & 0.037354 & 0.015916 & 0.00166274 & 0.84737 & 1.90126\\
$1\%$ & 0.017191 & 0.036223 & 0.015834 & 0.00156359 & 0.84831 & 1.92163\\
\bottomrule
\end{tabular}
\end{table}

\subsection{Example 3. Polynomial conductivity}
\label{subsec:num-example-three}

We next consider a polynomial conductivity of degree five on $\Omega$,
given by
\begin{equation}
\sigma_{\mathrm{true}}^{(3)}(x_1,x_2)=3-x_2-x_1^3x_2^2.
\label{eq:num-profile-three}
\end{equation}
On the closed square, its minimum is $1$ at $(1,1)$ and its maximum
is $5$ at $(-1,-1)$.
The true and computed conductivities are presented in
Figure~\ref{fig:numerical-profile-3}. We again use exact and $1\%$
noisy data, and list the reconstruction errors and computed ranges
in Table~\ref{tab:num-example-three}.

\numreconfigure{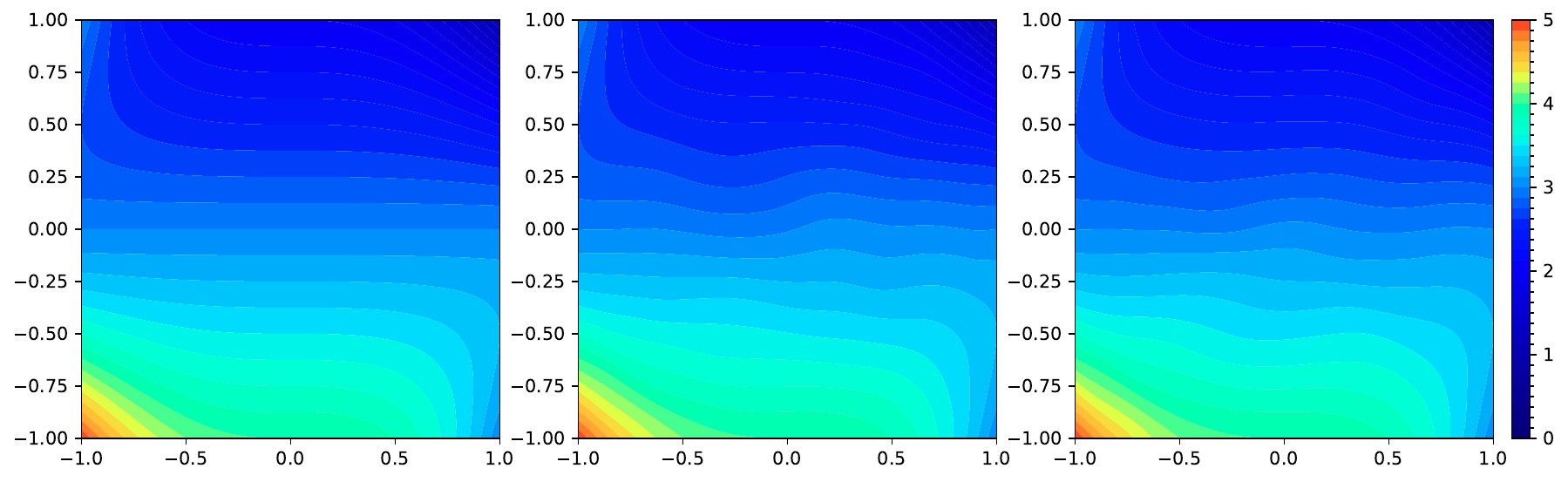}{Polynomial conductivity \eqref{eq:num-profile-three}}

\begin{table}[!htbp]
\centering\small
\setlength{\tabcolsep}{4pt}
\renewcommand{\arraystretch}{1.15}
\settowidth{\reconstructioncaptionwidth}{\normalfont\normalsize
Reconstruction errors and predicted conductivity ranges for exact and $1\%$ noisy data.}
\captionsetup{format=plain,font=normalsize,justification=raggedright,
  singlelinecheck=false,width=\reconstructioncaptionwidth}
\caption[Example 3. Reconstruction errors and predicted conductivity ranges for exact and $1\%$ noisy data.]{Example 3.\\
Reconstruction errors and predicted conductivity ranges for exact and $1\%$ noisy data.}
\label{tab:num-example-three}
\begin{tabular}{@{}lcccccc@{}}
\toprule
Data & $E_{\sigma,1}$ & $E_{\sigma,2}$ & $\mathrm{RE}_{\mathrm{mean}}$ & MSE & $\hat\sigma_{\min}$ & $\hat\sigma_{\max}$\\
\midrule
Exact & 0.003511 & 0.005107 & 0.003620 & 0.00024417 & 1.05860 & 4.94135\\
$1\%$ & 0.003207 & 0.004445 & 0.003216 & 0.00018498 & 1.05869 & 4.94104\\
\bottomrule
\end{tabular}
\end{table}

Figure~\ref{fig:numerical-profile-3} shows that both reconstructions closely match 
the true conductivity, reproducing the curved level sets and the decrease in conductivity 
from the lower left to the upper right.
Table~\ref{tab:num-example-three} quantifies this agreement, with relative $L^2$ errors of $0.005107$
and $0.004445$ for exact and noisy data, respectively.
The conductivity range is also accurately recovered. The reconstructed
minima and maxima are close to the true minimum and maximum on the
interior scoring grid, approximately $1.05901$ and $4.94099$, respectively.

\subsection{Example 4. Sinusoidal conductivity}
\label{subsec:num-example-four}

In the last example, we consider the sinusoidal conductivity on $\Omega$,
given by
\begin{equation}
\sigma_{\mathrm{true}}^{(4)}(x_1,x_2)
 =\frac32-\frac12\sin(2\pi x_1)\sin(2\pi x_2).
\label{eq:num-profile-four}
\end{equation}
Its values range from $1$ to $2$, with alternating regions of higher
and lower conductivity. This is the FourierBoard profile defined on the
unit square in \cite{KaloutEtAl2026}, rescaled to $\Omega=(-1,1)^2$.

Figure~\ref{fig:numerical-profile-4} presents
the true conductivity and the computed conductivities for exact and
$1\%$ noisy data. It demonstrates that the proposed method
accurately reconstructs the sinusoidal conductivity, resolving the
alternating regions of higher and lower conductivity from both exact and
noisy boundary data. The errors and computed ranges are listed in
Table~\ref{tab:num-example-four}. It shows relative $L^2$
errors of $0.025016$ and $0.039344$, respectively.

\numreconfigure{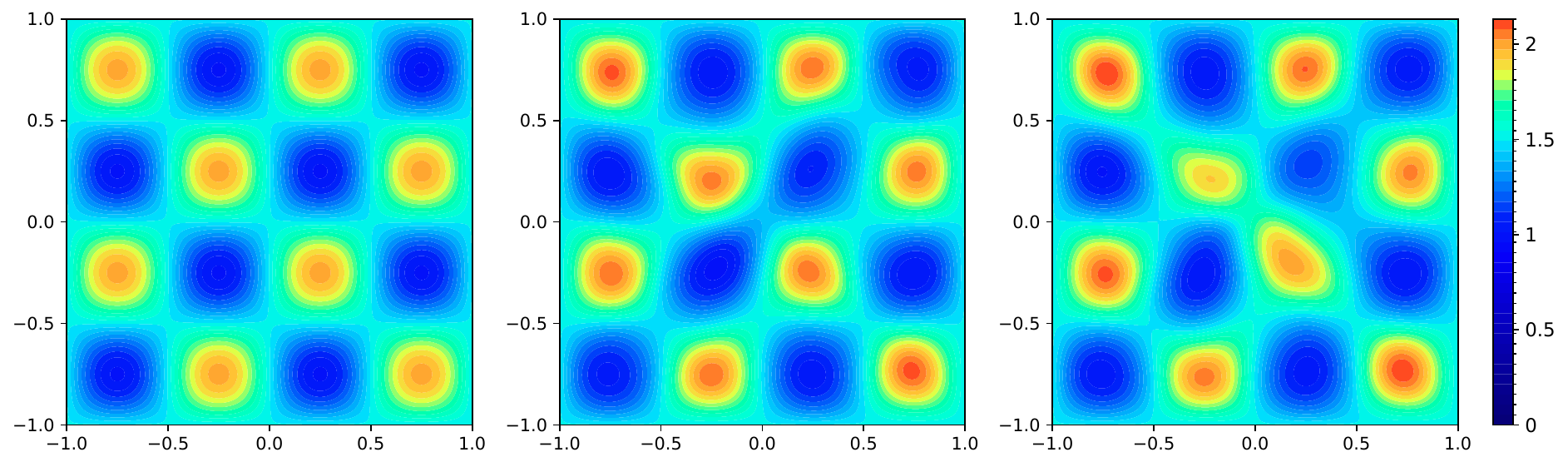}{Sinusoidal conductivity \eqref{eq:num-profile-four}}

\begin{table}[!htbp]
\centering\small
\setlength{\tabcolsep}{4pt}
\renewcommand{\arraystretch}{1.15}
\settowidth{\reconstructioncaptionwidth}{\normalfont\normalsize
Reconstruction errors and predicted conductivity ranges for exact and $1\%$ noisy data.}
\captionsetup{format=plain,font=normalsize,justification=raggedright,
  singlelinecheck=false,width=\reconstructioncaptionwidth}
\caption[Example 4. Reconstruction errors and predicted conductivity ranges for exact and $1\%$ noisy data.]{Example 4.\\
Reconstruction errors and predicted conductivity ranges for exact and $1\%$ noisy data.}
\label{tab:num-example-four}
\begin{tabular}{@{}lcccccc@{}}
\toprule
Data & $E_{\sigma,1}$ & $E_{\sigma,2}$ & $\mathrm{RE}_{\mathrm{mean}}$ & MSE & $\hat\sigma_{\min}$ & $\hat\sigma_{\max}$\\
\midrule
Exact & 0.018167 & 0.025016 & 0.018105 & 0.00144717 & 0.98373 & 2.10696\\
$1\%$ & 0.026564 & 0.039344 & 0.026314 & 0.00357972 & 1.00746 & 2.13258\\
\bottomrule
\end{tabular}
\end{table}

We also compare our results with the FourierBoard results in
\cite{KaloutEtAl2026}. Since \cite{KaloutEtAl2026} uses exact data, we use
our results for exact data in this comparison.
Table~\ref{tab:num-example-four-comparison} lists our results and the
FourierBoard results with and without Fourier feature encoding (FFE),
together with the numbers of boundary excitations, training epochs, and
trainable parameters.

\begin{table}[!htbp]
\centering\small
\setlength{\tabcolsep}{4pt}
\renewcommand{\arraystretch}{1.15}
\caption{Example 4. Comparison with the FourierBoard results in
\cite{KaloutEtAl2026} for exact data.}
\label{tab:num-example-four-comparison}
\begin{tabular}{@{}l!{\hspace{6pt}}rrrD{.}{.}{1.6}D{.}{.}{1.8}@{}}
\toprule
\multicolumn{1}{@{}c!{\hspace{6pt}}}{Method} &
\multicolumn{1}{c}{\shortstack{Boundary\\excitations}} &
\multicolumn{1}{c}{Epochs} &
\multicolumn{1}{c}{\shortstack{Trainable\\parameters}} &
\multicolumn{1}{c}{$\mathrm{RE}_{\mathrm{mean}}$} &
\multicolumn{1}{c@{}}{MSE}\\
\midrule
Our method & 4 & 150,000 & 47,305 & 0.018105 & 0.00144717\\
PINN with FFE~\cite{KaloutEtAl2026} & 140 & 400,300 & 611,714 & 0.084 & 0.025\\
PINN without FFE~\cite{KaloutEtAl2026} & 140 & 392,000 & 415,874 & 0.082 & 0.027\\
\bottomrule
\end{tabular}
\end{table}

The comparison in Table~\ref{tab:num-example-four-comparison} demonstrates
a substantial improvement over the FourierBoard results in
\cite{KaloutEtAl2026}. The mean pointwise relative error is approximately
$78\%$ smaller, and the MSE is approximately $94\%$ smaller. These improvements are achieved
using only four boundary excitations, compared with $140$ in
\cite{KaloutEtAl2026}. Our method also uses significantly fewer trainable parameters
than the method in \cite{KaloutEtAl2026}. For example, each of our joint
potential and stream function networks has three hidden layers with $64$
neurons per layer, whereas the potential network in \cite{KaloutEtAl2026},
shared across all excitations, has six hidden layers with $256$ neurons per
layer. Moreover, even our reconstruction obtained with $1\%$ measurement
noise has substantially smaller errors than both FourierBoard results for
exact data in \cite{KaloutEtAl2026}. These findings demonstrate
that our method achieves high reconstruction accuracy for an oscillatory
conductivity using a small number of boundary excitations.

\section{Concluding remarks}
\label{sec:conclusion}

In this paper, we develop a constrained neural stream function method 
for two dimensional EIT with Lipschitz conductivities. The stream 
representation enforces current conservation, reduces the conductivity 
equation to a first order system in strong form, and allows direct least 
squares minimization. The conductivity constraints and prescribed boundary 
voltages are imposed exactly through the neural trial functions. The measured 
currents are used to construct a fixed baseline for the stream functions, whose 
boundary traces remain free. The reconstruction uses only the boundary 
measurements and the governing equation.

We establish an error analysis relating the sampled neural loss to the
physical boundary current error. This estimate also yields a conductivity 
reconstruction estimate under the stated inverse stability condition.
The error estimates separate the
contributions of neural approximation, measurement noise, and sampling.
Numerical examples demonstrate high reconstruction accuracy with both
exact and noisy boundary data, and comparisons with other methods show
substantial reductions in reconstruction error.

There are several directions that can be investigated in future work.
In the present work, we use measurements on the whole boundary. We plan
to consider EIT problems with partial boundary data, where measurements
are available only on part of the boundary. Another direction for future
work is the reconstruction of piecewise continuous conductivities. This
problem lies outside the Lipschitz setting considered in this paper.
We plan to investigate suitable neural representations and develop the
corresponding approximation and reconstruction error analysis. We will
also investigate the number and choice of boundary excitations and
their effects on reconstruction accuracy with exact and noisy data.

\section*{Acknowledgements}

The author uses ChatGPT for language editing and assistance with optimizing the code. 
The author reviews and verifies the suggested changes and takes full responsibility for the content of the paper.

\section*{Data availability}

The data supporting the findings of this study will be available
after publication of the article.

\appendix

\section{Density results for the boundary factor}
\label{app:boundary-factors}

We prove Lemma~\ref{lem:rho-density} in this appendix.

\subsection{Uniform density}

We first establish
\eqref{eq:rho-density} for every boundary factor satisfying
\eqref{eq:boundary-factor-general}.
Functions in $\rho_\Omega C(\overline\Omega)$ vanish on
$\partial\Omega$, and so do their uniform limits. For the reverse
inclusion, let $h\in C(\overline\Omega)$ vanish on $\partial\Omega$ and
define
\[
v_\varepsilon
=\frac{h}{\max\{\rho_\Omega,\varepsilon\}}.
\]
For $\varepsilon>0$, we have $v_\varepsilon\in C(\overline\Omega)$ and
\[
|h-\rho_\Omega v_\varepsilon|
\leq |h|\,\mathbf 1_{\{\rho_\Omega<\varepsilon\}}.
\]
For any $\eta>0$, continuity of $h$ and $h|_{\partial\Omega}=0$ give
a neighborhood $U$ of $\partial\Omega$ such that
$|h|<\eta$ on $U\cap\overline\Omega$. Since $\rho_\Omega$ is continuous
and positive on the compact set $\overline\Omega\setminus U$, it has
a positive lower bound there. Hence, for sufficiently small $\varepsilon>0$,
$\{\rho_\Omega<\varepsilon\}\subset U$, and
\[
\|h-\rho_\Omega v_\varepsilon\|_{C(\overline\Omega)}
\leq\sup_{\{\rho_\Omega<\varepsilon\}}|h|
\leq\eta.
\]
Since $\eta>0$ is arbitrary, $\rho_\Omega v_\varepsilon\to h$ uniformly.
This proves \eqref{eq:rho-density}.

\subsection{\texorpdfstring{$H^2$}{H2} density}

We next prove \eqref{eq:rho-H2-density} for suitable boundary factors.
We first recall some notation. Let $C_c^\infty(\Omega)$ denote the space
of smooth functions compactly supported in $\Omega$, and set
\[
H_0^2(\Omega)
=\overline{C_c^\infty(\Omega)}^{\,H^2(\Omega)}.
\]
The space $C^m(\overline\Omega)$ consists of restrictions of
$C^m$ functions defined in a neighborhood of $\overline\Omega$.

For a boundary factor $\rho_\Omega$ satisfying
\eqref{eq:boundary-factor-general}, set
\[
Y_\rho
=\overline{\rho_\Omega C^\infty(\overline\Omega)}^{\,H^2(\Omega)}.
\]
Since $\rho_\Omega\in W^{2,\infty}(\Omega)$ and
$\rho_\Omega|_{\partial\Omega}=0$, every function in
$\rho_\Omega C^\infty(\overline\Omega)$ belongs to
$H^2(\Omega)\cap H_0^1(\Omega)$. Since
$H^2(\Omega)\cap H_0^1(\Omega)$ is closed in $H^2(\Omega)$, we obtain
\begin{equation*}
Y_\rho\subset H^2(\Omega)\cap H_0^1(\Omega).
\end{equation*}
It remains to choose $\rho_\Omega$ so that
\begin{equation}
H^2(\Omega)\cap H_0^1(\Omega)\subset Y_\rho.
\label{eq:appendix-reverse-inclusion}
\end{equation}
Once \eqref{eq:appendix-reverse-inclusion} is established, we have
$Y_\rho=H^2(\Omega)\cap H_0^1(\Omega)$. Together with the inclusions
\[
\rho_\Omega C^\infty(\overline\Omega)
\subset\rho_\Omega C^2(\overline\Omega)
\subset H^2(\Omega)\cap H_0^1(\Omega),
\]
this yields \eqref{eq:rho-H2-density}. We prove
\eqref{eq:appendix-reverse-inclusion} separately for the two domain classes.
We also need the following lemma for both cases.

\begin{lemma}
\label{lem:rho-H02-inclusion}
For every boundary factor $\rho_\Omega$ satisfying
\eqref{eq:boundary-factor-general}, we have
$H_0^2(\Omega)\subset Y_\rho$.
\end{lemma}

\begin{proof}
Let $\eta\in C_c^\infty(\Omega)$. Since $\rho_\Omega$ is bounded away
from zero on $\operatorname{supp}\eta$, the quotient
$\eta/\rho_\Omega$ belongs to $H^2(\Omega)$. Choose
$\phi_k\in C^\infty(\overline\Omega)$ such that
\[
\phi_k\to\frac{\eta}{\rho_\Omega}
\quad\text{in }H^2(\Omega).
\]
Multiplication by $\rho_\Omega\in W^{2,\infty}(\Omega)$ is bounded on
$H^2(\Omega)$, so $\rho_\Omega\phi_k\to\eta$ in $H^2(\Omega)$.
Thus $C_c^\infty(\Omega)\subset Y_\rho$. Since $Y_\rho$ is closed in
$H^2(\Omega)$, taking the $H^2(\Omega)$-closure gives
\[
H_0^2(\Omega)
=\overline{C_c^\infty(\Omega)}^{\,H^2(\Omega)}
\subset Y_\rho.
\]
\end{proof}

\subsubsection{Proof of \texorpdfstring{\eqref{eq:appendix-reverse-inclusion}}{(A.1)} for the \texorpdfstring{$C^{1,1}$}{C1,1} case}
\label{app:c11-density}

The signed distance has $C^{1,1}$ regularity in a tubular neighborhood
of the boundary, which gives the following
standard construction \cite{BoyerFabrie2013}.

\begin{lemma}
\label{prop:standard-c11-factor}
If $\partial\Omega$ is $C^{1,1}$, then there is a function
$\rho_\Omega\in C^{1,1}(\overline\Omega)$ such that
\begin{equation}
\rho_\Omega=0\quad\text{on }\partial\Omega,
\qquad
\rho_\Omega>0\quad\text{in }\Omega,
\qquad
\partial_n\rho_\Omega=-1\quad\text{on }\partial\Omega.
\label{eq:appendix-smooth-factor}
\end{equation}
\end{lemma}

We use this boundary factor in the following argument.
Let $u\in H^2(\Omega)\cap H_0^1(\Omega)$, and let
$\Tr_1=\partial_n$ denote the normal trace operator. By the trace theorem
\cite{Grisvard1985}, there is a bounded right inverse
\[
E_\partial\colon H^{1/2}(\partial\Omega)
\longrightarrow H^2(\Omega)\cap H_0^1(\Omega),
\qquad
\Tr_1E_\partial g=g.
\]
Choose $\phi_k\in C^\infty(\overline\Omega)$ such that
\[
\Tr\phi_k\to -\Tr_1u
\quad\text{in }H^{1/2}(\partial\Omega).
\]
Set $y_k=\rho_\Omega\phi_k$. Clearly, $u-y_k \in H^2(\Omega)\cap H_0^1(\Omega)$. 
In view of the boundary properties of
$\rho_\Omega$ in \eqref{eq:appendix-smooth-factor}, we have
\begin{equation}
\Tr_1y_k=-\Tr\phi_k\to\Tr_1u
\quad\text{in }H^{1/2}(\partial\Omega).
\label{eq:appendix-smooth-trace-convergence}
\end{equation}
Set
\[
r_k=\Tr_1(u-y_k),
\qquad
w_k=E_\partial r_k,
\qquad
z_k=u-y_k-w_k.
\]
Then $w_k\to0$ in $H^2(\Omega)$, and $z_k$ has zero Dirichlet and
normal traces. By the trace theorem, we have
$z_k\in H_0^2(\Omega)$. By Lemma~\ref{lem:rho-H02-inclusion},
$z_k\in Y_\rho$. Since $y_k\in Y_\rho$, we have
$u-w_k=y_k+z_k\in Y_\rho$. From $w_k\to0$ in $H^2(\Omega)$, we know
$u-w_k\to u$ in $H^2(\Omega)$.
Each $u-w_k$ belongs to $Y_\rho$, which is closed in $H^2(\Omega)$,
so $u\in Y_\rho$. Since $u\in H^2(\Omega)\cap H_0^1(\Omega)$ is
arbitrary, this proves \eqref{eq:appendix-reverse-inclusion}.

\subsubsection{Proof of \texorpdfstring{\eqref{eq:appendix-reverse-inclusion}}{(A.1)} for the curvilinear polygonal case}

In this case, the outward normal has different limits along the two incident edges
at each vertex. We therefore use an edgewise normal trace space.
We first recall this space and the corresponding trace theorem.
The edges are understood to be $C^{1,1}$ up to their endpoints,
and the interior angle at each vertex lies in $(0,\pi)$.
Let $V_1,\ldots,V_{N_\Gamma}$ be the vertices and write the open edges as
$\Gamma_j=\gamma_j((0,L_j))$, where $\gamma_j$ is parametrized by
arclength. Let $n_j$ denote the outward unit normal on $\Gamma_j$, and
let $\partial_{n_j}v$ denote the normal derivative of $v$ on $\Gamma_j$.
Define
\[
d_j(\gamma_j(s))=\min\{s,L_j-s\}
\]
and the Lions--Magenes space \cite{LionsMagenes2012} by
\[
H_{00}^{1/2}(\Gamma_j)
=\Big\{v\in H^{1/2}(\Gamma_j):
\int_{\Gamma_j}\frac{|v|^2}{d_j}\,ds<\infty\Big\},
\]
with norm
\[
\|v\|_{H_{00}^{1/2}(\Gamma_j)}^2
=\|v\|_{H^{1/2}(\Gamma_j)}^2
+\int_{\Gamma_j}\frac{|v|^2}{d_j}\,ds.
\]
Define
\[
G_\partial=\Big\{g: g|_{\Gamma_j}\in H_{00}^{1/2}(\Gamma_j), \ j=1,\ldots,N_\Gamma
\Big\},
\qquad
\|g\|_{G_\partial}^2
=\sum_{j=1}^{N_\Gamma}\|g|_{\Gamma_j}\|_{H_{00}^{1/2}(\Gamma_j)}^2.
\]
We use the following standard trace result from
\cite{Grisvard1985}.

\begin{lemma}
\label{prop:standard-polygon-traces}
For a convex $C^{1,1}$ curvilinear polygon, the map $\Tr_1v$, defined edgewise by
\[
(\Tr_1v)|_{\Gamma_j}
=(\partial_{n_j}v)|_{\Gamma_j},\qquad j=1,\ldots,N_\Gamma,
\]
is bounded from $H^2(\Omega)\cap H_0^1(\Omega)$ onto $G_\partial$ and
has a bounded right inverse
\[
E_\partial\colon G_\partial
\longrightarrow H^2(\Omega)\cap H_0^1(\Omega).
\]
If $\Tr_1v=0$, then $v\in H_0^2(\Omega)$. Moreover,
$C_c^\infty(\Gamma_j)$ is dense in
$H_{00}^{1/2}(\Gamma_j)$.
\end{lemma}

To prove \eqref{eq:appendix-reverse-inclusion}, it suffices to establish the following lemma.

\begin{lemma}
\label{claim:polygon-normal-density}
There exists a function $\rho_\Omega\in C^{1,1}(\overline\Omega)$
such that
\[
\rho_\Omega=0\quad\text{on }\partial\Omega,
\qquad
\rho_\Omega>0\quad\text{in }\Omega,
\]
and
\[
\overline{\left\{\Tr_1(\rho_\Omega\phi):
\phi\in C^\infty(\overline\Omega)\right\}}^{\,G_\partial}
=G_\partial.
\]
\end{lemma}

Before proving Lemma~\ref{claim:polygon-normal-density}, we show how it yields
\eqref{eq:appendix-reverse-inclusion}.
Let $\rho_\Omega$ be the boundary factor provided by this lemma,
and let $u\in H^2(\Omega)\cap H_0^1(\Omega)$. Choose
$\phi_k\in C^\infty(\overline\Omega)$ such that
\[
\Tr_1(\rho_\Omega\phi_k)\to\Tr_1u
\quad\text{in }G_\partial.
\]
Set $y_k=\rho_\Omega\phi_k$. Repeating the argument following
\eqref{eq:appendix-smooth-trace-convergence} in
Appendix~\ref{app:c11-density}, with $H^{1/2}(\partial\Omega)$ replaced
by $G_\partial$ and using the bounded right inverse from
Lemma~\ref{prop:standard-polygon-traces}, we obtain $u\in Y_\rho$.
This proves \eqref{eq:appendix-reverse-inclusion}.

\begin{proof}[Proof of Lemma~\ref{claim:polygon-normal-density}]

The construction in Appendix~\ref{app:c11-density} does not apply directly near the vertices, 
where the signed distance to $\partial\Omega$ is not $C^{1,1}$. We therefore use local signed 
distance functions that vanish on their respective edges and have outward normal 
derivative $-1$ there (see \eqref{eq:appendix-smooth-factor}).

Fix a vertex $V_m$ and extend its two incident $C^{1,1}$ edges
slightly beyond $V_m$. Let $\delta_m^{(1)}$ and $\delta_m^{(2)}$
be local signed distance defining functions for the extended edges,
chosen positive on the $\Omega$ side.
We choose $\delta_m^{(1)}\delta_m^{(2)}$ as the local factor near
$V_m$. It is $C^{1,1}$, positive in $\Omega$ near $V_m$, and
vanishes on both incident edges. 
For each $\Gamma_j$, let $U_j$ cover the portion of the edge outside
the vertex neighborhoods, and let $r_j$ be a local signed distance
defining function, positive on the $\Omega$ side.
Add an open interior set $U_0$ whose closure is a compact subset of
$\Omega$, so that these sets cover $\overline\Omega$.
Choose a nonnegative smooth partition of unity
$\{\eta_0,\eta_j,\eta_m^{\mathrm v}\}$ subordinate to this cover, with
each cutoff compactly supported in its corresponding neighborhood and
$\eta_m^{\mathrm v}=1$ near $V_m$. Define
\[
\rho_\Omega
=\eta_0+\sum_{j=1}^{N_\Gamma}\eta_jr_j
+\sum_{m=1}^{N_\Gamma}\eta_m^{\mathrm v}
\delta_m^{(1)}\delta_m^{(2)}.
\]
The terms $\eta_jr_j$ and
$\eta_m^{\mathrm v}\delta_m^{(1)}\delta_m^{(2)}$ are set equal to zero
outside their respective neighborhoods. Since the cutoffs vanish near
the boundaries of these neighborhoods, the extensions remain $C^{1,1}$.
Then we have
\[
\rho_\Omega\in C^{1,1}(\overline\Omega),
\qquad
\rho_\Omega=0\quad\text{on }\partial\Omega,
\qquad
\rho_\Omega>0\quad\text{in }\Omega.
\]

On $\Gamma_j$, set $a_j=-\partial_{n_j}\rho_\Omega$.
The function $a_j$ is Lipschitz and positive in the interior of
$\Gamma_j$. We next examine $a_j$ near each endpoint.
Suppose that $\Gamma_j$ lies on $\delta_m^{(1)}=0$ near an endpoint
$V_m$. Since $\eta_m^{\mathrm v}=1$ there,
\[
a_j
=\bigl(-\partial_{n_j}\delta_m^{(1)}\bigr)\delta_m^{(2)}
=\delta_m^{(2)}
\quad\text{on }\Gamma_j\text{ near }V_m.
\]
Since $\delta_m^{(2)}$ is the signed distance to the extension
of the other incident edge, its gradient at $V_m$ is the unit
normal to that extension, pointing toward the positive side
of $\delta_m^{(2)}$. We measure arclength along $\Gamma_j$
from $V_m$. Since the interior angle at $V_m$ lies in $(0,\pi)$,
the restriction of $\delta_m^{(2)}$ to $\Gamma_j$ has a strictly
positive derivative at $V_m$. Together with
$\delta_m^{(2)}(V_m)=0$, this shows that the restriction is
locally comparable to the arclength distance from $V_m$.
Thus $a_j$ is comparable to $d_j$ near each endpoint.  
We then have $0\leq a_j\leq C d_j$ on $\Gamma_j$. Thus, for
$v\in H^{1/2}(\Gamma_j)$,
\[
\int_{\Gamma_j}\frac{|a_jv|^2}{d_j}\,ds
\leq C\int_{\Gamma_j}d_j|v|^2\,ds
\leq C\|v\|_{L^2(\Gamma_j)}^2.
\]
Since $a_j$ is Lipschitz, multiplication by $a_j$ is bounded on
$H^{1/2}(\Gamma_j)$. Together with the preceding weighted estimate,
we have
\[
\|a_jv\|_{H_{00}^{1/2}(\Gamma_j)}^2
=\|a_jv\|_{H^{1/2}(\Gamma_j)}^2
+\int_{\Gamma_j}\frac{|a_jv|^2}{d_j}\,ds
\leq C\|v\|_{H^{1/2}(\Gamma_j)}^2.
\]
For any $\phi\in C^\infty(\overline\Omega)$, from $\rho_\Omega|_{\Gamma_j}=0$,
we know
\[
\partial_{n_j}(\rho_\Omega\phi)
=-a_j\Tr_{\Gamma_j}\phi \quad\text{on }\Gamma_j.
\]
Taking $v=-\Tr_{\Gamma_j}\phi$ in the preceding estimate, we obtain
$\Tr_1(\rho_\Omega\phi)\in G_\partial$.

To prove density, fix $g\in G_\partial$ and set
$g_j=g|_{\Gamma_j}$, $j=1,\ldots,N_\Gamma$.
For each $j$, choose $g_{j,k}\in C_c^\infty(\Gamma_j)$ such that
\[
g_{j,k}\to g_j
\quad\text{in }H_{00}^{1/2}(\Gamma_j).
\]
Since $a_j$ is Lipschitz and bounded away from zero on
$\operatorname{supp}g_{j,k}$, the function $h_{j,k}=-g_{j,k}/a_j$
belongs to $H^{1/2}(\Gamma_j)$ and has compact support in
$\Gamma_j$. Define $h_k$ on $\partial\Omega$ edgewise by
$h_k|_{\Gamma_j}=h_{j,k}$.
So $h_k\in H^{1/2}(\partial\Omega)$.
By the density of smooth boundary traces, choose
$\phi_k\in C^\infty(\overline\Omega)$ such that
\[
\|\Tr\phi_k-h_k\|_{H^{1/2}(\partial\Omega)}<\frac1k.
\]
Define $g^{(k)}$ on $\partial\Omega$ edgewise by $g^{(k)}|_{\Gamma_j}=g_{j,k}$. Using
$g_{j,k}=-a_jh_{j,k}$, we obtain
\[
\begin{aligned}
&\|\Tr_1(\rho_\Omega\phi_k)-g^{(k)}\|_{G_\partial}^2
=\sum_{j=1}^{N_\Gamma}
\|a_j(\Tr_{\Gamma_j}\phi_k-h_{j,k})\|_{H_{00}^{1/2}(\Gamma_j)}^2\\
&\quad\leq C\sum_{j=1}^{N_\Gamma}
\|\Tr_{\Gamma_j}\phi_k-h_{j,k}\|_{H^{1/2}(\Gamma_j)}^2
\leq C\|\Tr\phi_k-h_k\|_{H^{1/2}(\partial\Omega)}^2
<\frac{C}{k^2}.
\end{aligned}
\]
Since $g^{(k)}\to g$ in $G_\partial$, it follows that
$\Tr_1(\rho_\Omega\phi_k)\to g$ in $G_\partial$.
\end{proof}

\subsection{Convex polygons with straight edges}

For an ordinary convex polygon, the boundary factor in the preceding argument can be chosen as 
\[\rho_\Omega=\prod_{j=1}^{N_\Gamma}\ell_j,\]
where $\ell_j$ is an affine function that vanishes on the $j$th side and
is positive in $\Omega$. Near each vertex, this product consists of the two incident affine factors multiplied by a smooth positive factor, and therefore has the local properties used above. For $\Omega=(-1,1)^2$, this gives
\[
\rho_\Omega(x_1,x_2)=(1-x_1^2)(1-x_2^2).
\]

\section{Sampling and quadrature estimates}
\label{app:uniform-sampling}

We give the proofs of Proposition~\ref{prop:uniform-sampling},
Proposition~\ref{prop:det-midpoint}, and
Proposition~\ref{prop:sampling-selection}.
For $\Theta\in\mathcal K_R$, we introduce the integrands
\begin{align*}
F_\Omega(\Theta,x)
&=\frac1M\sum_{i=1}^M|s_i(W^\Theta)(x)|^2,\\
F_b(\Theta,y)
&=\frac{\lambda_b}{M}\sum_{i=1}^M
|\Tr s_i(W^\Theta)(y)|^2,\\
F_q^\delta(\Theta,z)
&=\frac{\lambda_q}{M}\sum_{i=1}^M
|d_i^\delta(W^\Theta)(z)|^2.
\end{align*}
With this notation, we have
\begin{align*}
\J^\delta(\Theta)
&=\int_\Omega F_\Omega(\Theta,x)\,dx
+\int_{\partial\Omega}F_b(\Theta,y)\,ds
+\int_{\partial\Omega}F_q^\delta(\Theta,z)\,ds,\\
\J_{\mathbf N}^\delta(\Theta)
&=\mathcal I_{\Omega,N_\Omega}[F_\Omega(\Theta,\cdot)]
+\mathcal I_{b,N_b}[F_b(\Theta,\cdot)]
+\mathcal I_{q,N_q}[F_q^\delta(\Theta,\cdot)].
\end{align*}

\medskip

\begin{proof}[\textbf{\textit{Proof of Proposition~\ref{prop:uniform-sampling}}}]
We fix the neural network architecture and $\delta$.

(1) Since the parameter box is compact,
the raw network outputs and their spatial derivatives through order two
are bounded uniformly in $\Theta$. From the constrained representations, we obtain
\begin{align*}
F_\Omega(\Theta,x)
&\leq C_R\Big[1+\sum_{i=1}^M
\bigl(|\nabla F_i(x)|^2+|\nabla B_i(x)|^2\bigr)\Big],\\
F_b(\Theta,y)
&\leq C_R\Big[1+\sum_{i=1}^M
\bigl(|\Tr\nabla F_i(y)|^2+|\Tr\nabla B_i(y)|^2\bigr)\Big],\\
F_q^\delta(\Theta,z)
&\leq C_R\Big[1+\sum_{i=1}^M
\bigl(|\Tr\nabla B_i(z)|^2
+|g_i^\delta(z)|^2\bigr)\Big].
\end{align*}
The bounds on the right are integrable. Indeed, $F_i,B_i\in H^2(\Omega)$,
their first derivatives have $L^2$ boundary traces, and
$g_i^\delta\in\mathcal H^{1/2}(\partial\Omega)\hookrightarrow
L^2(\partial\Omega)$. Moreover, at a fixed spatial point, $F_\Omega$, $F_b$, and
$F_q^\delta$ are continuous in $\Theta$. We may therefore apply the uniform
strong law on a compact parameter space \cite{Jennrich1969} separately to
these three functions. This gives
\begin{align*}
\sup_{\Theta\in\mathcal K_R}
\Big|
\mathcal I_{\Omega,N_\Omega}[F_\Omega(\Theta,\cdot)]
-\int_\Omega F_\Omega(\Theta,x)\,dx
\Big|&\to 0,\\
\sup_{\Theta\in\mathcal K_R}
\Big|
\mathcal I_{b,N_b}[F_b(\Theta,\cdot)]
-\int_{\partial\Omega}F_b(\Theta,y)\,ds
\Big|&\to 0,\\
\sup_{\Theta\in\mathcal K_R}
\Big|
\mathcal I_{q,N_q}[F_q^\delta(\Theta,\cdot)]
-\int_{\partial\Omega}F_q^\delta(\Theta,z)\,ds
\Big|&\to 0
\end{align*}
almost surely as $N_\Omega$, $N_b$, and $N_q$, respectively, tend to infinity.
Since there are only three limits, they hold simultaneously on a common event
of probability one. This proves \eqref{eq:proved-uniform-sampling}.

\medskip

(2) We next establish the expectation and variance estimates for a fixed
$\Theta\in\mathcal K_R$.
Since $x^{(1)},\ldots,x^{(N_\Omega)}$ are uniformly distributed in $\Omega$, we have
\[
\mathbb E[F_\Omega(\Theta,x^{(j)})]
=\frac1{|\Omega|}\int_\Omega F_\Omega(\Theta,x)\,dx,
\qquad
\mathbb E[|F_\Omega(\Theta,x^{(j)})|^2]
=\frac1{|\Omega|}\int_\Omega |F_\Omega(\Theta,x)|^2\,dx.
\]
By the definition of $\mathcal I_{\Omega,N_\Omega}$ and independence of the
interior samples, we obtain
\begin{align*}
&\mathbb E\bigl[\mathcal I_{\Omega,N_\Omega}[F_\Omega(\Theta,\cdot)]\bigr]
=\int_\Omega F_\Omega(\Theta,x)\,dx,\\
&\operatorname{Var}\bigl(\mathcal I_{\Omega,N_\Omega}[F_\Omega(\Theta,\cdot)]\bigr)
=\frac{|\Omega|^2}{N_\Omega^2}
\sum_{j=1}^{N_\Omega}\operatorname{Var}\bigl(F_\Omega(\Theta,x^{(j)})\bigr)\\
&\quad\leq\frac{|\Omega|^2}{N_\Omega^2}
\sum_{j=1}^{N_\Omega}\mathbb E[|F_\Omega(\Theta,x^{(j)})|^2]
=\frac{|\Omega|}{N_\Omega}\|F_\Omega(\Theta,\cdot)\|_{L^2(\Omega)}^2.
\end{align*}
Applying the same argument on the boundary gives
\begin{align*}
\mathbb E\bigl[\mathcal I_{b,N_b}[F_b(\Theta,\cdot)]\bigr]
&=\int_{\partial\Omega}F_b(\Theta,y)\,ds,
&
\mathbb E\bigl[\mathcal I_{q,N_q}[F_q^\delta(\Theta,\cdot)]\bigr]
&=\int_{\partial\Omega}F_q^\delta(\Theta,z)\,ds,\\
\operatorname{Var}\bigl(\mathcal I_{b,N_b}[F_b(\Theta,\cdot)]\bigr)
&\leq\frac{|\partial\Omega|}{N_b}\|F_b(\Theta,\cdot)\|_{L^2(\partial\Omega)}^2,
&
\operatorname{Var}\bigl(\mathcal I_{q,N_q}[F_q^\delta(\Theta,\cdot)]\bigr)
&\leq\frac{|\partial\Omega|}{N_q}\|F_q^\delta(\Theta,\cdot)\|_{L^2(\partial\Omega)}^2.
\end{align*}
The above expectation identities and linearity of expectation prove
the expectation identity in \eqref{eq:mc-fixed-moments}. 

For the variance, the preceding estimates for each term
and the Cauchy--Schwarz inequality give
\begin{align*}
&\quad\operatorname{Var}\bigl(\J_{\mathbf N}^\delta(\Theta)\bigr)
\leq 3\Bigl(
\operatorname{Var}\bigl(\mathcal I_{\Omega,N_\Omega}[F_\Omega(\Theta,\cdot)]\bigr)
+\operatorname{Var}\bigl(\mathcal I_{b,N_b}[F_b(\Theta,\cdot)]\bigr)
+\operatorname{Var}\bigl(\mathcal I_{q,N_q}[F_q^\delta(\Theta,\cdot)]\bigr)
\Bigr)\\
&\leq 3\Bigl(
\frac{|\Omega|}{N_\Omega}\|F_\Omega(\Theta,\cdot)\|_{L^2(\Omega)}^2
+\frac{|\partial\Omega|}{N_b}\|F_b(\Theta,\cdot)\|_{L^2(\partial\Omega)}^2
+\frac{|\partial\Omega|}{N_q}\|F_q^\delta(\Theta,\cdot)\|_{L^2(\partial\Omega)}^2
\Bigr)\\
&\leq \frac{3|\Omega|}{MN_\Omega}\sum_{i=1}^M
\|s_i(W^\Theta)\|_{L^4(\Omega)}^4
+\frac{3\lambda_b^2|\partial\Omega|}{MN_b}\sum_{i=1}^M
\|\Tr s_i(W^\Theta)\|_{L^4(\partial\Omega)}^4+\frac{3\lambda_q^2|\partial\Omega|}{MN_q}\sum_{i=1}^M
\|d_i^\delta(W^\Theta)\|_{L^4(\partial\Omega)}^4.
\end{align*}
We now bound the $L^4$ norms on the right uniformly in $\Theta$.
By the constrained representations and compactness of $\mathcal K_R$, we have
\begin{align*}
s_i(W^\Theta)\in H^1(\Omega;\mathbb R^2),\qquad
\Tr s_i(W^\Theta)\in H^{1/2}(\partial\Omega;\mathbb R^2),\qquad
d_i^\delta(W^\Theta)\in\mathcal H^{1/2}(\partial\Omega),
\end{align*}
with norms bounded uniformly in $\Theta$. The Sobolev embeddings in the
two dimensional domain and on its one dimensional boundary then yield
the required $L^4$ bounds.
Consequently,
\[
\operatorname{Var}\bigl(\J_{\mathbf N}^\delta(\Theta)\bigr)
\leq C_{R,\delta}
\left(N_\Omega^{-1}+N_b^{-1}+N_q^{-1}\right),
\]
which proves the variance bound in \eqref{eq:mc-fixed-moments}.

\medskip

(3) Finally, for fixed sample points, each summand in $\J_{\mathbf N}^\delta$
is continuous in $\Theta$. Thus $\J_{\mathbf N}^\delta$ is continuous on
the compact set $\mathcal K_R$ and attains a minimum.
\end{proof}

\begin{proof}[\textbf{\textit{Proof of Proposition~\ref{prop:det-midpoint}}}]
Under the stated boundary regularity, we may choose the fixed liftings 
$F_i,B_i\in W^{2,\infty}(\Omega)$. Fix $\Theta\in\mathcal K_R$. The $C^2$
regularity of the raw network outputs, together with
$\rho_\Omega\in C^{1,1}(\overline\Omega)$,
$A_b\in C^{0,1}(\overline\Omega)$, and the regularity of the liftings, implies that
$\sigma^\vartheta$, $\nabla u_i^{\theta_i}$, and
$\nabla\psi_i^{\theta_i}$ are Lipschitz in the spatial variable. Hence
$F_\Omega(\Theta,\cdot)$ is Lipschitz in $\Omega$ and
$F_b(\Theta,\cdot)$ is Lipschitz in boundary arclength. Moreover,
$\partial_\tau\psi_i^{\theta_i}$ is Lipschitz in arclength. Together with
$g_i^\delta\in C^{0,1}(\partial\Omega)$, this implies that $F_q^\delta(\Theta,\cdot)$ 
is Lipschitz in arclength. In the curvilinear
polygonal case, each boundary assertion is understood edgewise.

We next show that these spatial Lipschitz bounds are uniform in $\Theta$.
Since the neural network architecture is finite and $\mathcal K_R$ is
compact, we know the network outputs and their spatial derivatives through
order two are uniformly bounded over $\mathcal K_R$. The preceding Lipschitz constants are
therefore bounded by constants $L_\Omega$, $L_b$, and $L_q$, independent
of $\Theta$. For each cell $K\in\mathcal T_\Omega$ with quadrature point $x^{(K)}$,
\[
\left|
\int_K F_\Omega(\Theta,x)\,dx
-|K|F_\Omega(\Theta,x^{(K)})
\right|
\leq L_\Omega|K|\operatorname{diam}(K)\leq L_\Omega|K|h_\Omega.
\]
For boundary intervals with quadrature points $y^{(E)}$ and $z^{(E)}$,
respectively, the analogous estimates are
\[
\begin{aligned}
\left|
\int_E F_b(\Theta,y)\,ds
-|E|F_b(\Theta,y^{(E)})
\right|
&\leq L_b|E|^2\leq L_b|E|h_b,
&&E\in\mathcal E_b,\\
\left|
\int_E F_q^\delta(\Theta,z)\,ds
-|E|F_q^\delta(\Theta,z^{(E)})
\right|
&\leq L_q|E|^2\leq L_q|E|h_q,
&&E\in\mathcal E_q.
\end{aligned}
\]
Summing these local estimates and applying the triangle inequality,
uniformly in $\Theta$, yields
\[
\varepsilon_{\mathrm{sam},\mathbf N}^\delta
\leq L_\Omega|\Omega|h_\Omega
+L_b|\partial\Omega|h_b
+L_q|\partial\Omega|h_q
\leq C(h_\Omega+h_b+h_q).
\]
Applying the assumed bounds on $h_\Omega$, $h_b$, and $h_q$ proves
\eqref{eq:deterministic-sampling-consistency}.

Finally, the sampled loss is continuous in $\Theta$ and therefore attains
a minimum on the compact set $\mathcal K_R$.
\end{proof}

\begin{proof}[\textbf{\textit{Proof of Proposition~\ref{prop:sampling-selection}}}]
We first consider Monte Carlo sampling. For every fixed $n$,
Proposition~\ref{prop:uniform-sampling} gives
\[
\varepsilon_{\mathrm{sam},n,\mathbf N}^{\delta_n}\to0
\qquad\text{in probability as }
\min\{N_\Omega,N_b,N_q\}\to\infty.
\]
For the threshold $n^{-1}$, we can make the probability of exceedance
smaller than $2^{-n}$ by taking all three point counts sufficiently large.
We therefore choose a triple
$\mathbf N_n=(N_{\Omega,n},N_{b,n},N_{q,n})$ such that
\[
\min\{N_{\Omega,n},N_{b,n},N_{q,n}\}\geq n
\qquad\text{and}\qquad
\mathbb P\!\left(
\varepsilon_{\mathrm{sam},n,\mathbf N_n}^{\delta_n}>n^{-1}
\right)\leq2^{-n}.
\]
The preceding bound gives
\[
\sum_{n=1}^\infty
\mathbb P\!\left(
\varepsilon_{\mathrm{sam},n,\mathbf N_n}^{\delta_n}>n^{-1}
\right)
\leq\sum_{n=1}^\infty2^{-n}<\infty.
\]
By the Borel--Cantelli lemma,
\[
\mathbb P\!\left(
\varepsilon_{\mathrm{sam},n,\mathbf N_n}^{\delta_n}>n^{-1}
\ \text{infinitely often}
\right)=0.
\]
Hence, almost surely, $\varepsilon_{\mathrm{sam},n,\mathbf N_n}^{\delta_n}\leq n^{-1}$
for all sufficiently large $n$, and therefore
$\varepsilon_{\mathrm{sam},n,\mathbf N_n}^{\delta_n}\to0$ almost surely.

For deterministic quadrature, we apply Proposition~\ref{prop:det-midpoint}
separately for each $n$. Once the neural class and the data at level $n$
are fixed, the constant $C$ in
\eqref{eq:deterministic-sampling-consistency}, now denoted as $C_n$, is
independent of the point counts. We may therefore choose a triple 
$\mathbf N_n=(N_{\Omega,n},N_{b,n},N_{q,n})$ such that
\[
\min\{N_{\Omega,n},N_{b,n},N_{q,n}\}\geq n
\qquad\text{and}\qquad
C_n\bigl(
N_{\Omega,n}^{-1/2}+N_{b,n}^{-1}+N_{q,n}^{-1}
\bigr)\leq n^{-1}.
\]
We then have $\varepsilon_{\mathrm{sam},n,\mathbf N_n}^{\delta_n}\leq n^{-1}$,
which proves deterministic convergence of the sampling error to zero.
\end{proof}

\section{Construction of the liftings}
\label{app:numerical-details}

In this appendix, we give the formulas used to construct the liftings in
Section~\ref{subsec:numerical-networks}.

\subsection{Harmonic lifting for the potential}
\label{app:num-potential-harmonic}

For boundary data $f$, let $p_f$ be the bilinear corner interpolant
obtained from \eqref{eq:numerical-corner-interpolant} by replacing
$\sigma_b$ with $f$. Subtract $p_f$ from $f$ on each side and denote
the resulting residuals by $r_{f,e}$, $e\in\{L,R,B,T\}$. Since $p_f$ agrees with $f$ at
the four corners, each residual vanishes at both endpoints.
We extend each residual harmonically into the square, with zero boundary
values on the other three sides. Adding these four extensions to $p_f$
gives the harmonic lifting $\mathcal H[f]$. Indeed, $p_f$ is itself
harmonic, and the resulting sum agrees with $f$ on each side. The standard
harmonic lifting is given by
\begin{equation*}
\begin{aligned}
\mathcal H[f](x_1,x_2)={}&p_f(x_1,x_2)\\
&+\sum_{m=1}^{\infty}\sin(m\pi\xi)
 \left[a_{f,B,m}\frac{\sinh(m\pi(1-\zeta))}{\sinh(m\pi)}
      +a_{f,T,m}\frac{\sinh(m\pi\zeta)}{\sinh(m\pi)}\right]\\
&+\sum_{m=1}^{\infty}\sin(m\pi\zeta)
 \left[a_{f,L,m}\frac{\sinh(m\pi(1-\xi))}{\sinh(m\pi)}
      +a_{f,R,m}\frac{\sinh(m\pi\xi)}{\sinh(m\pi)}\right],
\end{aligned}
\end{equation*}
where 
\[
\xi=\frac{x_1+1}{2},\quad \zeta=\frac{x_2+1}{2},\quad \text{and}\quad
a_{f,e,m}=2\int_0^1 r_{f,e}(t)\sin(m\pi t)\,dt.
\]
We set $F_i=\mathcal H[f_i]$ to lift each potential $f_i$
($i=1,2,3,4$).
In the numerical implementation, we truncate each series at $K=1024$
and denote the resulting approximation by $\mathcal H_K[f_i]$.
We approximate the sine coefficients by
\begin{equation*}
a_{f,e,m}\approx\frac{2}{N_{\mathrm{sin}}}
 \sum_{j=1}^{N_{\mathrm{sin}}-1}r_{f,e}(j/N_{\mathrm{sin}})\sin(m\pi j/N_{\mathrm{sin}}),
\qquad N_{\mathrm{sin}}=8192.
\end{equation*}
At interior points, we evaluate $F_i$ using $\mathcal H_K[f_i]$.
On the boundary, we use $f_i$ and its tangential derivative directly
and compute the normal derivatives using the truncated harmonic series.

\subsection{Current interpolation and stream lifting}
\label{app:num-stream-trace}
\label{app:num-harmonic}

The current measurements are available only at the points of $Q$.
To construct the boundary stream trace, we first interpolate these data
along each side. On each side, let $t\in[0,2]$ denote arclength in the counterclockwise
direction. The measured points are $t_k=(k-\tfrac12)\Delta s$,
$k=1,\ldots,n_s$, where $n_s=200$ and $\Delta s=2/n_s$.
For each experiment $i=1,2,3,4$, we define $\ell_{i,\epsilon}$ on each
edge by linear interpolation between the current samples on
$[t_k,t_{k+1}]$. On each of the two endpoint half
intervals, we extend the interpolant constantly using the nearest sample
value. We enforce compatibility through a correction that preserves
the measurements. For this purpose, we introduce the fixed function
\begin{equation*}
\phi(t)=
\begin{cases}
\sin^2\!\left(\frac{\pi(t-t_k)}{\Delta s}\right),
 &t_k\leq t\leq t_{k+1},\quad k=1,\ldots,n_s-1,\\
0,&0\leq t\leq t_1\ \text{or}\ t_{n_s}\leq t\leq2.
\end{cases}
\end{equation*}
This function vanishes at every measured point and at the corners.
We define
\begin{equation*}
\widetilde g_{i,\epsilon}=\ell_{i,\epsilon}+c_i\phi,\quad\text{with}\quad
c_i=-\frac{\int_{\partial\Omega}\ell_{i,\epsilon}\,ds}
          {\int_{\partial\Omega}\phi\,ds}
   =-\frac{\Delta s\sum_{z\in Q}g_{i,\epsilon}(z)}
                {4(n_s-1)\Delta s/2}
   =-\frac{1}{2(n_s-1)}\sum_{z\in Q}g_{i,\epsilon}(z).
\end{equation*}
The resulting current $\widetilde g_{i,\epsilon}$ then has
zero total boundary integral and retains the value $g_{i,\epsilon}$
at every point of $Q$.

Let $\gamma:[0,8]\to\partial\Omega$ be the counterclockwise arclength 
parametrization starting at $x_\star=(-1,-1)$. The boundary
stream trace is
\begin{equation*}
\zeta_{i,\epsilon}(\gamma(s))
   =\int_0^s\widetilde g_{i,\epsilon}(\gamma(t))\,dt.
\end{equation*}
We evaluate this integral from the analytic primitives of the linear
interpolants and squared sine functions. The zero total flux ensures
that $\zeta_{i,\epsilon}$ is a continuous trace on the closed boundary.
We construct the stream lifting $B_{i,\epsilon}$ by applying the procedure
in Appendix~\ref{app:num-potential-harmonic} with $f=\zeta_{i,\epsilon}$.
At interior points, we use $\mathcal H_K[\zeta_{i,\epsilon}]$. On the boundary,
we prescribe $B_{i,\epsilon}=\zeta_{i,\epsilon}$ and
$\partial_\tau B_{i,\epsilon}=\widetilde g_{i,\epsilon}$ directly. The
latter agrees with $g_{i,\epsilon}$ at the measured midpoints. We compute
the normal derivatives on the boundary using the truncated harmonic series.
The data loss uses the original measurements $g_{i,\epsilon}$.

\end{document}